\documentclass[11pt]{amsart}
\usepackage{hyperref}

\usepackage{enumerate,color}
\usepackage{amssymb}

\usepackage[utf8]{inputenc}
\usepackage[T1]{fontenc}

\usepackage[mathscr]{eucal}
\usepackage{amssymb}

\usepackage[a4paper, hmargin={3cm,3cm}, centering]{geometry}

\everymath{\displaystyle}

\everymath{\displaystyle}

\newtheorem{lemma}{Lemma}[section]
\newtheorem{theorem}[lemma]{Theorem}
\newtheorem*{theorem*}{Theorem}
\newtheorem{corollary}[lemma]{Corollary}

\newtheorem{proposition}[lemma]{Proposition}

\newtheorem*{proposition*}{Proposition}

\theoremstyle{remark}

\theoremstyle{definition}
\newtheorem*{definition*}{Definition}
\newtheorem*{conjecture*}{Conjecture}
\newtheorem*{remark*}{Remark}
\newtheorem*{remarks*}{Remarks}
\newtheorem*{claim*}{Claim}

\begin{document}

\title[Pointwise multiple averages for sublinear functions]{Pointwise multiple averages for sublinear functions}

\author{Sebasti{\'a}n Donoso, Andreas Koutsogiannis and Wenbo Sun}
\address[Sebasti{\'a}n Donoso]{Universidad de O'Higgins, Instituto de Ciencias de la Ingenier\'ia, Rancagua, Chile}
\email{sebastian.donoso@uoh.cl}
\address[Andreas Koutsogiannis]{The Ohio State University, Department of mathematics, Columbus, Ohio, USA} \email{koutsogiannis.1@osu.edu}
\address[Wenbo Sun]{The Ohio State University, Department of mathematics, Columbus, Ohio, USA} \email{sun.1991@osu.edu}

\thanks{The first author is supported by Fondecyt Iniciaci\'on en Investigaci\'on Grant 11160061.}

\begin{abstract} 
For any  measure preserving system $(X,\mathcal{B},\mu,T_1,\ldots,T_d),$ where we assume no commutativity on the transformations $T_i,$ $1\leq i\leq d,$ we study the pointwise convergence of multiple ergodic averages with iterates of different growth coming from a large class of sublinear functions. This class properly contains important subclasses of Hardy field functions of order $0$ and of Fej{\'e}r functions, i.e., tempered functions of order $0.$ We show that the convergence of the single average, via an invariant property, implies the convergence of the multiple one. 
We also provide examples of sublinear functions which are in general bad for convergence on arbitrary systems, but they are good for  uniquely ergodic systems. The case where the fastest function is linear is addressed as well, and we provide, in all the cases, an explicit formula of the limit function.
\end{abstract}



\maketitle

\section{Introduction and main results}\label{se:1}

 The study of the limiting behavior, in $L^2(\mu)$ or pointwise, as $N\to\infty,$ of  {\it multiple ergodic averages} of the form 
\begin{equation}\label{E:central}
\frac{1}{N}\sum_{n=0}^{N-1}T_1^{a_1(n)}f_1 \cdot\ldots\cdot T_d^{a_d(n)}f_d, \end{equation}
where $T_1,\ldots,T_d:X\to X$ are invertible (usually commuting) measure preserving transformations acting on a probability space $(X,\mathcal{B},\mu)$; $f_1,\ldots,f_d\in L^\infty(\mu)$ and $a_i(\cdot)$ are suitable functions taking integer values on integers for all $1\leq i\leq d,$\footnote{For a measurable function $f$ and a transformation $T\colon X\to X$,  $Tf$ denotes the composition $f\circ T$. 
} 
 is a central problem in ergodic theory. With partial knowledge of the limiting behavior of (\ref{E:central}), for the case $T_i=T$ and $a_i(n)=in$, Furstenberg provided a purely ergodic theoretical proof of Szemer{\'e}di's
theorem (\cite{Fu}), i.e., every subset of $\mathbb{N}$ with positive upper density contains arbitrary long arithmetic progressions. 

In recent years, motivated by the work of Furstenberg,
 fruitful progress has been made towards the study of the existence and also of the exact value of the $L^{2}(\mu)$ limit of (\ref{E:central}) for various classes of functions $a_i$. For the existence of the limit, we refer the readers to \cite{Aus, HK05, H, Tao, Walsh12}. As for the explicit expression of the limit, the first result is von Neumann's mean ergodic theorem which says that for $d=1$ and $a_1(n)=n$ the limit of \eqref{E:central} is $\mathbb{E}(f_{1}\vert\mathcal{I}(T_{1}))$, where $\mathcal{I}(T)$ denotes the $\sigma$-algebra of $T$-invariant sets and $\mathbb{E}(f\vert\mathcal{I}(T))$ is the \emph{conditional expectation} of $f$ with respect to $\mathcal{I}(T).$ The classes of integer polynomial, integer parts of real polynomial, Hardy field  (see definition in $\S~2$) and more generally tempered classes of functions are also studied in \cite{Be87a, BK, F2} for a single $T$ and in \cite{CFH, F3, K} for commuting $T_i$'s.

 We stress out a result due to Frantzikinakis which is closely related to our study.  In \cite[Theorem~2.7]{F2}, under no commutativity assumption on $T_i$'s, for integer part of functions $a_i$ in $\mathcal{LE},$\footnote{$a$ is a \emph{logarithmico-exponential Hardy field function} if it belongs to a Hardy field of real valued functions and it's defined on some $(c,+\infty),$ $c\geq 0,$ by a finite combination of symbols $+, -, \times, \div, \sqrt[n]{\cdot}, \exp, \log$ acting on the real variable $x$ and on real constants (for more on Hardy field functions and in particular for logarithmico-exponential ones check for example \cite{F2, F3}).} with $x^\varepsilon\prec a_d\prec\ldots\prec a_1\prec x,$\footnote{For two functions $a,$ $b$ we write $a(x)\prec b(x),$ or just $a\prec b$ if $a(x)/b(x)\to 0 $  as $x\to\infty.$} for some $\varepsilon>0,$ he showed  that the limit of \eqref{E:central} in $L^{2}(\mu)$ is equal to $ \mathbb{E}(f_1|\mathcal{I}(T_1))\cdot\ldots\cdot \mathbb{E}(f_d|\mathcal{I}(T_d)).$

\medskip

On the other hand, the general problem of pointwise convergence of multiple ergodic averages remains widely open. For the existence and explicit expression of the pointwise limit of (\ref{E:central}) very few things are known. Even the $d=1$ case is not completely understood (for some results see \cite{BKQW,Bo2}).  
For $d\geq 2$,  Bourgain  showed (in \cite{Bo}) that the pointwise limit of (\ref{E:central}) exists when $T_1=T_2$ and $a_1(n)=an$, $a_2(n)=bn$ for $a,b\in \mathbb{Z}$. 
Recently, Huang, Shao and Ye (\cite{HSY}) showed the existence of the pointwise limit of \eqref{E:central} for $T_{i}=T, a_{i}(n)=in$ under the assumption that $T$ is a distal transformation (see also \cite{GHSY} for some particular weakly mixing systems). This result was extended in \cite{DS0} to two commuting transformations generating a distal action, and for an arbitrary number of commuting transformations in \cite{DS}  (also for a distal system and linear iterates). Note that when we deal with multiple $T_i$'s we have to impose additional \emph{distinctness} conditions on the $a_i$'s, since in general we don't expect \eqref{E:central} to have a nice behavior (for the case of $d=2$ and $a_1(n)=a_2(n)=n$ see a counterexample in \cite[Section~4]{BL}  when $T_1,$ $T_2$ generate a solvable group).

\medskip

In this paper, we study the pointwise convergence of \eqref{E:central} for integer part  (denoted with $[\cdot]$) of sequences of functions of different growth rate which are at most linear. More specifically, for a wide class  of sublinear functions, $\mathcal{S^{\ast}},$ we show (see next subsection for notation):

\begin{theorem}\label{C:main}  Let $d\in \mathbb{N}$ and $(X_i,\mu_i,T_i),$ $1\leq i\leq d$ be measure preserving systems. Let $a_i\in \mathcal{S^{\ast}},$ $1\leq i\leq d$ with $a_d\prec \ldots\prec a_1$ and $a'_d\prec\ldots\prec a'_1,$ and $\nu$ be any coupling of the spaces $(X_i,\mu_i)$. 
Then the averages \[ \frac{1}{N}\sum_{n=0}^{N-1} T_{1}^{[a_1(n)]}f_1(x_1) \cdot\ldots \cdot  T_{d}^{[a_d(n)]}f_d(x_d) \] converge as $N\to\infty$ for $\nu$-a.e. $(x_1,\ldots,x_d)\in X_1\times\cdots\times X_d$ to  \[ \mathbb{E}(f_{1}\vert\mathcal{I}(T_{1}))(x_1)\cdot\ldots\cdot\mathbb{E}(f_{d}\vert\mathcal{I}(T_{d}))(x_d).\]

\noindent In particular, if $(X_i,\mu_i)=(X,\mu),$ $1\leq i \leq d$, we have that \[ \frac{1}{N}\sum_{n=0}^{N-1} T_{1}^{[a_1(n)]}f_1(x) \cdot\ldots \cdot  T_{d}^{[a_d(n)]}f_d(x) \] converge as $N\to\infty$ for $\mu$-a.e. $x\in X$ to  \[ \mathbb{E}(f_{1}\vert\mathcal{I}(T_{1}))(x)\cdot\ldots\cdot\mathbb{E}(f_{d}\vert\mathcal{I}(T_{d}))(x).\] 
\end{theorem}

\begin{remark*} Let $\mathcal{LE}_\varepsilon$ denote the set of logarithmico-exponential Hardy field functions $a$   satisfying the growth condition $x^\varepsilon\prec a(x)\prec x$ for some $\varepsilon>0.$
 By a variation of the argument in \cite[Proposition~6.4]{F2}, one can obtain a different proof of Theorem \ref{C:main} for the special case where $a_{i}\in\mathcal{LE}_\varepsilon,$ $1\leq i\leq d$.  The idea of \cite[Proposition~6.4]{F2} is to convert the multiple averages for sublinear functions of different growth in $\mathcal{LE}_\varepsilon,$ via a change of variable, to an average of the same form but with linear fastest function. Our method, which is applicable into a larger class of functions, has a different philosophy focusing instead on the invariance property of the averages under the transformations $T_{1}\times id\times\dots\times id,$ $id\times T_{2}\times\dots\times id,$ $\dots,$ $id\times\dots\times id\times T_{d}$, via which we deduce the limit of the expressions of interest. Another advantage of this method is that it can also be used to show that there are certain sublinear functions for which even though the pointwise convergence might in general fail, it holds on all the uniquely ergodic systems (see $\S$~\ref{sec:ue} for details).
\end{remark*}

When $a_{1}$ is linear, i.e., a polynomial of degree 1, we have:

\begin{theorem} \label{Thm:Linear}
	Let $d\in \mathbb{N},$ $(X_i,\mu_i,T_i)$, $1\leq i \leq d$  be measure preserving systems, $a_1$ be a linear function, $a_i\in \mathcal{S^{\ast}},$ $2\leq i\leq d$ with $a_d\prec \ldots\prec a_1$ and $a'_d\prec\ldots\prec a'_1,$ and $\nu$ be any coupling of the spaces $(X_i,\mu_i)$. Then  for all $f_i  \in L^{\infty}(\mu_i),$ $1\leq i\leq d$  
	the averages \[ \frac{1}{N}\sum_{n=0}^{N-1} T_{1}^{[a_1(n)]}f_1(x_1) T_{2}^{[a_2(n)]}f_2(x_2)\cdot\ldots\cdot  T_{d}^{[a_d(n)]}f_d(x_d) \] converge as $N\to\infty$ for $\nu$-a.e. $\vec{x}=(x_1,\ldots,x_d)\in X_1\times\cdots \times X_d.$
	
In particular, if $a_{1}(n)=kn+\ell$, $k=p/q,\; p,\;q\in\mathbb{Z}\setminus\{0\}$, then the limit is equal to
	  \[ \frac{1}{q}\sum_{j=0}^{q-1}\mathbb{E}\Big(T_{1}^{[\frac{pj}{q}+\ell]}f_{1}\vert\mathcal{I}(T^{p}_{1})\Big)(x_1)\mathbb{E}(f_{2}\vert\mathcal{I}(T_{2}))(x_2)\cdot\ldots\cdot\mathbb{E}(f_{d}\vert\mathcal{I}(T_{d}))(x_d),\]
while if $a_{1}(n)=\gamma n+\ell,\gamma\in\mathbb{R}\backslash\mathbb{Q}$, then the limit is equal to
 \[ F(x_1)\mathbb{E}(f_{2}\vert\mathcal{I}(T_{2}))(x_2)\cdot\ldots\cdot\mathbb{E}(f_{d}\vert\mathcal{I}(T_{d}))(x_d),\]
where 
$$F(x)=\sum_{m\in\mathbb{Z}}\exp\Big(2\pi i\frac{m\ell}{\gamma}\Big)\cdot\frac{\exp(-2\pi i\frac{m}{\gamma})-1}{-2\pi i \frac{m}{\gamma}}\mathbb{E}(f_{1}\vert \mathcal{I}_{\gamma,m}(T))(x)$$
and $\mathcal{I}_{\gamma,m}(T)$ is the sub-$\sigma$-algebra generated by the eigenspace of $T$ with eigenvalue $-\frac{m}{\gamma}$.
\end{theorem}

\noindent Via Theorems~\ref{C:main} and ~\ref{Thm:Linear}, we immediately get the following result on sequences of different growth of the form $(n^c)_{n},$ $0<c\leq 1$ (which also follows by methods of Frantzikinakis, using a variation of \cite[Proposition~6.4]{F2}):

\begin{corollary}\label{C:main1}
Let $d\in\mathbb{N}$ and $(X,\mu,T_1,\ldots,T_d)$ be a measure preserving system. For all $0<c_d<\ldots<c_1\leq 1$ and  $f_1,\ldots,f_d \in L^{\infty}(\mu),$ the averages \[ \frac{1}{N}\sum_{n=0}^{N-1} T_{1}^{[n^{c_1}]}f_1(x)\cdot\ldots\cdot T_{d}^{[n^{c_d}]}f_d(x) \] converge as $N\to\infty$ for $\mu$-a.e. $x\in X$ to  \[ \mathbb{E}(f_{1}\vert\mathcal{I}(T_{1}))(x)\cdot\ldots\cdot\mathbb{E}(f_{d}\vert\mathcal{I}(T_{d}))(x).\] 
\end{corollary}

\subsection{Single convergence implies multiple convergence}

The philosophy of this article is that for a specific nice and wide class of sublinear functions, we have that ``single convergence implies the multiple one''. More specifically, assuming no commutativity on the transformations $T_i$,  we show in Theorem~\ref{C:main}, that averages as in \eqref{E:central} with integer parts of functions of different growth rate from the aforementioned class, converge pointwisely and the limit is the expected one, i.e., the product of conditional expectations, using the fact that the single average converges. 
The same method also extends to the case when the fastest growing function is linear and we get Theorem~\ref{Thm:Linear}. Our arguments throughout the article are elementary and have a soft touch of ergodic theory.

\medskip

At this point we introduce some notation.  
For $1\leq i\leq d,$ let $(X_i,\mu_i,T_i)$ be measure preserving (m.p.) systems (we also assume that each $(X_i,T_i)$ is a topological dynamical system) and $\mu_i=\int \mu_{[T_i],x}d\mu_i(x)$ be the disintegration of $\mu_i$ over its factor $\mathcal{I}(T_i)$ (i.e., the {\em ergodic decomposition}).  For $\vec{x}=(x_1,\ldots,x_d)$, let $\mu_{[T_{1},\dots,T_{d}], \vec{x}}$ be the measure on $X_1\times\cdots\times X_d$ defined by
$$\int\limits_{X_1\times \cdots \times X_d} f_{1}\otimes\dots\otimes f_{d}\; d\mu_{[T_{1},\dots,T_{d}],\vec{x}}=\mathbb{E}(f_{1}\vert\mathcal{I}(T_{1}))(x_1)\cdot\ldots\cdot\mathbb{E}(f_{d}\vert\mathcal{I}(T_{d}))(x_d)$$ 
for all $f_{1},\dots,f_{d}\in L^{\infty}(\mu).$
It is easy to see that $\mu_{[T_{1},\dots,T_{d}],\vec{x}}=\bigotimes_{i=1}^d \mu_{[T_i],x_i}.$ 
 Let also
\begin{equation}\label{E:(1)}
\lambda_{N,\vec{x}}:=\frac{1}{N}\sum_{n=0}^{N-1}\Big(T_{1}^{[a_1(n)]}\times\dots\times T_{d}^{[a_d(n)]}\Big)\delta_{\vec{x}},
\end{equation}
where  $\vec{x} \in X_1\times\cdots\times X_d$ and $\delta_{\vec{x}}$ denotes the Dirac measure at $\vec{x}$. 

Denoting with $\mathbb{R}^+$ a set of the form $(c,+\infty)$ for some $c\geq 0,$ we define the class
\begin{eqnarray*}
S:= \left\{a \in \mathcal{C}^{3}(\mathbb{R}^+)\;\Big\vert\; a,\; \frac{1}{a'}\in \mathcal{SL} \;\text{and}\;a^{-1}\in M_1\cap D_0\cap D_1\cap (D_2\cup M_2) 
  \right\},
\end{eqnarray*}
where  for $k\in\mathbb{N}\cup\{0\}$ $$D_k:=\left\{a:\;
\limsup_{x\to \text{sgn}(a^{-1})\cdot\infty}\sup\limits_{h\in [-1,1] }\left|\frac{a^{(k+1)}(x+h)}{a^{(k)}(x)}\right|<\infty\right\},$$ $$M_k:=\left\{a:\;a^{(k)}\;\text{is eventually monotone}\right\},$$  sgn is the \emph{sign} function and $$\mathcal{SL}=\{a:\;a(x)\prec x\}$$ is the set of \emph{sublinear functions} (recall that $a\prec b$ means $a(x)/b(x)\to 0$ as $x\to\infty$).

\medskip

Note that the $\limsup$ that appears in the definition of $D_k$ can in general be any $\alpha\in [0,\infty]. $
Indeed, for $k=0,$ the function $a(x)=\log x$ gives $\alpha=0$; to get a specific $\alpha>0,$ pick $\beta$ with $\beta\exp(\beta)=\alpha$ and let $a(x)=\exp(\beta x),$  while to get $\alpha=\infty$ pick $a(x)=\exp(x^2)$. 

\medskip

Note also that every function $a\in \mathcal{S}$ satisfies $\log x\prec a(x).$ Indeed, since $a^{-1}\in M_1,$ we have that $a'$ has eventually constant sign, hence integrating the relation $x|a'(x)|\geq M$ (that holds eventually for $M>0$ since $1/a'\in \mathcal{SL}$) we get $\log x\ll |a(x)|.$ Using again that $1/a'\in\mathcal{SL}$ and since $\log x$ isn't fast enough to have this property, we get $\log x\prec a(x).$

Let $\mathcal{S}^{\ast}\subseteq \mathcal{S}$ denote the subclass of functions where  $\lim_{N\to\infty}\frac{1}{N}\sum_{n=0}^{N-1} f(T^{[a(n)]}x)$ exists pointwisely (a.e.) for every measure preserving system $(X,\mu,T)$ and every bounded measurable function $f$.  We stress out the fact that $\mathcal{S}^\ast$ is a strict subset of $\mathcal{S}$ (see $\S$~\ref{sec:ue}). 








\medskip

The following result via a density argument will lead us to the proof of Theorem~\ref{C:main}.

\begin{theorem} \label{Thm:LimitMeasure}
	Let $d\in \mathbb{N}$ and $(X_i,\mu_i,T_i),$ $1\leq i\leq d$ be m.p. systems, $a_i\in \mathcal{S^{\ast}},$ $1\leq i\leq d$ with $a_d\prec \ldots\prec a_1$ and $a'_d\prec\ldots\prec a'_1,$ and $\nu$ be any coupling of the spaces $(X_i,\mu_i)$. Then, for $\nu$-a.e. $\vec{x} \in X_1\times \cdots \times X_d$,   we have that   $\lambda_{N,\vec{x}}$ converges to $\mu_{[T_1,\ldots,T_d],\vec{x}}$ as $N\to\infty.$ 
\end{theorem}

\begin{remark*}  In $\S$~\ref{se:2}, we show that $\mathcal{S}$ contains functions  $a$ which belong to some Hardy field and satisfy  $x^\varepsilon\prec a(x)\prec x$ for some $\varepsilon>0.$ So, by  \cite[Theorem~3.4]{BKQW} we actually have that each such function $a$ is in $\mathcal{S}^{\ast}$ (with convergence to the expected limit, i.e., the conditional expectation $\mathbb{E}(f|\mathcal{I}(T))$) while slow Hardy field functions (i.e., $1\prec a(x)\prec \log x \exp((\log(\log x))^m)$ for some $0\leq m< 1$) don't belong to $\mathcal{S^{\ast}}$ (see \cite[Theorem~3.6]{BKQW}). However, even though Theorem~\ref{Thm:LimitMeasure} in general might fail (take for example $a_i\in \mathcal{S}\backslash\mathcal{S^{\ast}}$ for some $1\leq i\leq d$), we have its validity for uniquely ergodic systems 
 and continuous functions on them, since then  {single convergence holds not only for functions in $\mathcal{S}^{\ast}$ but for all functions in $\mathcal{S}$} 
 (see Theorem~\ref{D:sUniqueErgodic}).
\end{remark*}

\subsection{Pointwise averages on uniquely ergodic systems}\label{sec:ue}

A topological system  $(X,T)$ is  \emph{uniquely ergodic} if there is a unique Borel probability measure which is $T$-invariant. Via the following result for single convergence, which we show in $\S$~\ref{Sec:MainResults}, under the unique ergodicity assumption of the system, we extend Theorem~\ref{Thm:LimitMeasure} (to Theorem~\ref{D:mainUniqueErgodic} below).



\begin{theorem}\label{D:sUniqueErgodic} Let $(X,T)$ be a uniquely ergodic system with unique $T$-invariant measure $\mu$ and $a\in \mathcal{S}.$ Then for any continuous function $f$ on $X,$ for every $x\in X$ 
	we have that \[ \lim_{N\to\infty}\frac{1}{N}\sum_{n=0}^{N-1} T^{[a(n)]}f(x)=\int_{X}f\,d\mu.\]
\end{theorem}

As we mentioned before, this result has been studied for general systems in \cite{BKQW,Bo2} along functions which belong to a smaller class of functions than $\mathcal{S}$ (see Theorem \ref{thm:BKQW2}), and in general it might fail for $a\in \mathcal{S}.$ One can show (after some elementary calculations) that $$a(x)=\log x \log(\log x)\in S\setminus S^\ast$$ ($a\notin S^\ast$ by \cite[Theorem~3.6]{BKQW}, since $a$ satisfies the hypothesis of the slow growth rate).


\noindent A similar argument as in Theorem~\ref{C:main} extends Theorem~\ref{D:sUniqueErgodic} to multiple averages:

\begin{theorem}\label{D:mainUniqueErgodic} Let $d\in\mathbb{N}$, $(X_i,T_i)$ be uniquely ergodic systems with unique $T_{i}$-invariant measures $\mu_i,$ $a_i\in \mathcal{S},$ $1\leq i\leq d$ with $a_d\prec \ldots\prec a_1$ and $a'_d\prec\ldots\prec a'_1,$ and $f_i$ be continuous functions on $X_i$, $1\leq i\leq d$.  Then for any coupling $\nu$ of these systems, 
	the averages \[ \frac{1}{N}\sum_{n=0}^{N-1} T_{1}^{[a_1(n)]}f_1(x_1) \cdot\ldots \cdot  T_{d}^{[a_d(n)]}f_d(x_d) \] converge as $N\to\infty$ for $\nu$-a.e. $(x_1,\ldots,x_d)\in X_1\times\cdots\times X_d$ to  \[ \int f_{1}\;d\mu_1 \cdot\ldots\cdot\int f_{d}\;d\mu_d.\] 
\end{theorem}

\begin{remark*} For pointwise averages with $a_{1}$ being a linear function, a similar statement to Theorem~\ref{Thm:Linear} can be derived assuming unique ergodicity of $(X_i,\mu_i,T_i)$ for $i\geq 2$ taking averages over functions in $\mathcal{S}$ on continuous functions (the unique ergodicity of $(X_1,\mu_1,T_1)$ is not necessary since for the linear $a_{1}$ we always have single convergence).
\end{remark*}

 In $\S$~\ref{se:2}, we introduce a specific large family $\mathcal{T}$ of sublinear functions which contains properly the Hardy field functions $a$ with $x^\varepsilon\prec a(x)\prec x$ for some $\varepsilon>0$ and is contained properly in the class of Fej{\'e}r functions (see $\S$~\ref{se:2} for definitions of these two important classes of functions).
	We also show (in $\S$~\ref{se:2} and $\S$~\ref{se:4}) that $\mathcal{T}$ is properly contained in $\mathcal{S}^{\ast}$ and that for functions of different growth rate in $\mathcal{T}$ we get different growth rate for their derivatives; as a result, our results hold for $a_{i}\in\mathcal{T}$ of different growth (see Corollary \ref{cor:1} for details).

\medskip

\subsection*{Definitions  and notations.}
A {\em measure preserving system} $(X,\mathcal{X},\mu, T_1,\ldots,T_d)$ is a probability space $(X,\mathcal{X},\mu)$ endowed with measure preserving transformations $T_i\colon X\to X$, meaning that $\mu(T_i^{-1}A)=\mu(A)$ for all $A\in \mathcal{X}$ and $1\leq i\leq d.$ We omit writing the $\sigma$-algebra $\mathcal{X}$ when this causes no confusion.  Throughout the paper we always assume that $X$ is a compact metric space, $\mathcal{X}$ is its Borel $\sigma$-algebra and $\mu$ is a Borel measure. We let $M(X)$ denote the convex set of probability measures on $X$ which is compact for the weak$^{\ast}$-topology.  
A {\em coupling} $\lambda$ of two probability spaces $(X_1,\mu_1)$ and $(X_2,\mu_2)$ is a measure in $M(X\times Y)$, whose marginals are equal to $\mu_1$ and $\mu_2$ respectively.
A {\em joining} of two measure preserving systems $(X_1,\mu_1,T_1,\ldots,T_d)$ and $(X_2,\mu_2,S_1,\ldots,S_d)$ is a coupling of $(X_1,\mu_1)$ and $(X_2,\mu_2)$ that is invariant under the diagonal transformations $T_1\times S_1,\ldots, T_d \times S_d$ (these definitions extend naturally to $k$ systems, $k\geq 2$). 
A {\em factor map} between two measure preserving systems $(X,\mu,T_1,\ldots,T_d)$ and $(Y,\nu,S_1,\ldots,S_d)$ is a measurable function $\pi\colon X\to Y$ such that the push-forward measure $\pi_{\ast}\mu$ is equal to  $\nu$ and  $\pi\circ T_i=S_i\circ \pi$, $1\leq i\leq d$. 
 Finally, for two quantities $a$ and $b$, we write $a\ll b$ if there exists $c>0$ such that $|a|\leq c\cdot |b|,$  and $a\ll_{\delta_1,\ldots,\delta_r} b$ if there exists $c\equiv c(\delta_1,\ldots,\delta_r)>0$ with $ |a|\leq c\cdot |b|.$

\subsection*{ Acknowledgements.} We thank Nikos Frantzikinakis for bringing us to the alternative proof of Theorem~\ref{C:main} for functions of different growth from the class $\mathcal{LE}_{\varepsilon}$.

\section{A nice class of sublinear functions}\label{se:2}

 In this section we define a nice class of sublinear functions, $\mathcal{T},$ first appeared in \cite{BK},  which we'll show that is a subclass of $\mathcal{S}^{\ast}$ defined in the previous section. More specifically, we show in this section that $\mathcal{T}$ is a proper subset of $\mathcal{S}$ and then, in $\S$~\ref{se:4}, we prove that $\mathcal{T}$ is a proper subset of $\mathcal{S}^{\ast}.$

Let $$\mathcal{R}:=\Big\{a\in \mathcal{C}^3(\mathbb{R}^+):\;\text{the limits}\;\lim_{x\to\infty}\frac{xa'(x)}{a(x)},\;\lim_{x\to\infty}\frac{xa''(x)}{a'(x)},\;\text{and}\;\lim_{x\to\infty}\frac{xa'''(x)}{a''(x)}\;\text{exist in}\;\mathbb{R}\Big\}$$ and
 $$\mathcal{T}:=\Big\{a\in \mathcal{R}:\;\lim_{x\to\infty}\frac{xa'(x)}{a(x)}\in (0,1)\;\text{or}\;\lim_{x\to\infty}\frac{xa'(x)}{a(x)}=1 \;\text{and}\;\lim_{x\to\infty}a'(x)=0\;\text{monotonicaly}\Big\}.$$ 

\medskip

 
 We start with the connection between  $\mathcal{T}$ and some important classes of sublinear functions.
 
\medskip

\subsubsection{Fej{\'e}r functions}  A function $a\in C^1((c,\infty)), c\geq 0,$ is a \emph{Fej{\'e}r} function if:
	(i) $a'(x)$ tends monotonically to $0$ as $x\to\infty;$ and
	(ii)  $\lim_{x\to\infty}x|a'(x)|=\infty$.	
	We denote with $\mathcal{F}$ the set of all Fej{\'e}r functions. Note that every $a\in \mathcal{F}$ is eventually monotone and satisfies the growth rate conditions  $\log x\prec  a(x)\prec x$ (See \cite{BK} for more details on Fej{\'e}r functions).
	
\medskip

\subsubsection{Hardy field functions}  Let $B$ be the collection of equivalence classes of real valued functions defined on some halfline $(c,\infty),$ $c\geq 0,$ where two functions that agree eventually are identified. These equivalence classes are called \emph{germs} of functions.  A \emph{Hardy field} is a subfield of the ring $(B, +, \cdot)$ that is closed under differentiation.\footnote{We use the word \emph{function} when we refer to elements of $B$ (understanding that all the operations defined and statements made for elements of $B$ are considered only for sufficiently large values of $x\in\mathbb{R}^+$).} (See \cite{F3} for more details on Hardy field functions.)

If $\mathcal{H}$ is the union of all Hardy fields, every element of $\mathcal{H}$ has eventually constant sign, from which it follows that if $a\in\mathcal{H},$ then $a$ is eventually monotone and the limit $\lim_{x\to\infty}a(x)$  exists (possibly infinite--as in the $\mathcal{LE}_\varepsilon$ class of functions). For functions $a,\;b\in\mathcal{H}$ with $b\neq 0,$ it follows that the asymptotic growth ratio $\lim_{x\to\infty} a(x)/b(x)$ exists (possibly infinite), fact that will often justify the use of L’Hospital’s rule. 
For some $\varepsilon>0$ we denote with $\mathcal{H}^s_\varepsilon$ the set of functions $a$ which belongs to some Hardy field $\mathcal{H}$ and satisfy $x^\varepsilon\prec a(x)\prec x.$\footnote{We say that the Hardy field functions $a$ which satisfy $\log x\prec a(x)\prec x$ are of \emph{polynomial degree 0}.}

\begin{remark*}
$\mathcal{H}^s_\varepsilon$ is a proper subset of $\mathcal{T}$ and $\mathcal{T}$ is a proper subset of $\mathcal{F}.$ 

Indeed, if $a\in \mathcal{H}^s_\varepsilon,$ then $a'(x)\to 0$ monotonically and  $$\lim_{x\to\infty}\frac{xa'(x)}{a(x)}=\lim_{x\to\infty}\frac{\log|a(x)|}{\log x}\in [\varepsilon,1],$$ so $a\in \mathcal{T}$ (the other limits exist by the properties of a Hardy field). If $a\in \mathcal{T}$ and $\lim_{x\to\infty}\frac{xa'(x)}{a(x)}\in (0,1),$ then $a\in\mathcal{F}$ from \cite[Lemma~2.1]{BK}, while if $\lim_{x\to\infty}\frac{xa'(x)}{a(x)}=1$ and $a'(x)\to 0$ monotonically, then $a\in \mathcal{F}$ by \cite[Lemma~2.2]{BK}.

Since $\log^\alpha x\in \mathcal{F}\setminus\mathcal{T}$ for all $\alpha>1,$ and $a(x)=x^{1/2}(2+\cos\sqrt{\log x})\in \mathcal{T}\setminus \mathcal{H}^s_\varepsilon$ for all  $\varepsilon>0$ (the derivative of $a^2$ is not (eventually) monotone by \cite[Section~1]{BK}), the claim follows. 
\end{remark*}

The following lemma provides a sufficient condition for a function to belong to $D_0$ (recall the definition from $\S$~\ref{se:1}).

\begin{lemma}\label{L:new condition}
Let $a\in C^1(\mathbb{R})$ and $\alpha,$ $\beta\in \mathbb{R}$ with $\alpha\leq \frac{xa'(x)}{a(x)}\leq \beta$ eventually. Then for any $H>0$ and large enough $x,$ we have that  \begin{equation}\label{E:sup}\sup\limits_{h\in [-H,H] } \left|\frac{a'(x+h)}{a(x)}\right|\ll_{\alpha,\beta}|x|^{\beta-\alpha-1}.
\end{equation}
\end{lemma}

\begin{proof}
Since $\alpha\leq \frac{xa'(x)}{a(x)}\leq \beta$ eventually, following the proof of \cite[Lemma 2.2]{BK} we have that there exist positive constants $C_\alpha$ and $C_\beta$ such that 
	$ C_\alpha |x|^{\alpha} \leq | a(x) | \leq C_\beta |x|^{\beta}.  $
Using the fact that eventually we have $\left|\frac{a'(x)}{a(x)}\right|\leq \frac{\max\{|\alpha|,|\beta|\}}{|x|},$ we get
$$ \sup_{h\in [-H,H]}\left|\frac{a'(x+h)}{a(x)}\right|\leq \max\{|\alpha|,|\beta|\} C_\alpha^{-1}C_\beta\cdot |x|^{\beta-\alpha-1}\max\left\{|1-H/x|^{\beta-1},\;|1+H/x|^{\beta-1}\right\}$$ and the result follows.
\end{proof}

\begin{remark*}
If the constants $\alpha,$ $\beta$ of Lemma~\ref{L:new condition} satisfy $\beta-\alpha-1<0$ (a special case of this is when $\lim_{x\to \pm \infty}\frac{xa'(x)}{a(x)}\in\mathbb{R}$) then the limit of \eqref{E:sup} exists (as $x\to \infty$ or $-\infty$) and it is $0.$
\end{remark*}


Since $\mathcal{T}\subseteq\mathcal{F}$, every $a\in \mathcal{T}$ is sublinear, (eventually) monotone, and has the property that $\frac{1}{a'(x)}\prec x.$ Moreover, we have:

\begin{proposition}\label{P:psub}
$\mathcal{T}$ is a proper subset of $\mathcal{S}.$
\end{proposition}
\begin{proof}
Assuming that $a\in \mathcal{T}$ is eventually positive (the other case follows analogously), using the fact that $\lim_{x\to\infty}\frac{xa'(x)}{a(x)}=\alpha\in (0,1],$ we get $\lim_{x\to\infty}\frac{xa''(x)}{a'(x)}=\alpha-1,\;\text{and} \;\lim_{x\to\infty}\frac{xa'''(x)}{a''(x)}=\alpha-2$ (the properties of $\mathcal{T}$ allows us to use L'Hopital's rule), so we have that
$$\lim_{x\to\infty}\frac{x(a^{-1})'(x)}{a^{-1}(x)}=\lim_{y\to\infty}\frac{a(y)}{ya'(y)}=\frac{1}{\alpha},\;\;\lim_{x\to\infty}\frac{x(a^{-1})''(x)}{(a^{-1})'(x)}=-\lim_{y\to\infty}\frac{a''(y)a(y)}{(a'(y))^2}=\frac{1}{\alpha}-1,\;\;\text{and}$$ $\lim_{x\to\infty}\frac{x(a^{-1})'''(x)}{(a^{-1})''(x)}=\lim_{y\to\infty}\Big(\frac{a(y)a'''(y)}{a'(y)a''(y)}-3\frac{a(y)a''(y)}{(a'(y))^2}\Big)=\frac{1}{\alpha}-2.$ The previous remark implies that $a\in\mathcal{S}.$
Since $a(x)=\log x\log(\log x)\in \mathcal{S}\setminus \mathcal{T}$ ($a$ satisfies all the properties of $S$ -- we skip all the elementary calculations -- and $xa'(x)/a(x)\to 0$ as $x\to\infty$), we have the claim.
\end{proof}

 \begin{remark*}
The class $\mathcal{T}$ misses not only slow functions as $a_1(x)=\log x \log(\log x)$ from $S$ (i.e., functions $a$ with $1\prec a\prec x^\varepsilon$ for all $\varepsilon>0$) but also functions as $$a_2(x)=x^{\alpha}(4/\alpha+\sin{\log x})^3$$ for $0<\alpha<1/20$ for which we have $a_2\notin \mathcal{T}$ since the ratio $xa_2'(x)/a_2(x)$ doesn't have a limit as $x\to\infty$ (it is bounded though between positive numbers since $\alpha-\frac{3\alpha}{4-\alpha}\leq \frac{xa_2'(x)}{a_2(x)}\leq \alpha+\frac{3\alpha}{4-\alpha}$) and $a_2$ is not slow since $x^{\alpha/2}\prec a_2(x)$. Since $a_2(x)\prec x,$ $a_2\in M_1\subseteq M_0$ with $a_2'(x)\to 0$ and $1/a_2'(x)\prec x,$ we actually have that $a_2\in \mathcal{F}\setminus \mathcal{T}.$  We also have that $a_2^{-1}\in D_0\cap D_1\cap D_2$ (the $\limsup$ that appear in these sets are limits and equal to $0$ -- we skip all the elementary calculations), hence $a_2(x)\in \mathcal{S}\setminus \mathcal{T}$.

 We will show in $\S$~\ref{se:4} that $a_2$ is a special function as $a_2\in S^\ast.$ 
\end{remark*}

The following lemma informs us that for a function in $D_0$ we have that the ratios of horizontal translations are bounded.

\begin{lemma}\label{L:bound}
Let $a\in \mathcal{C}^1(\mathbb{R})$ and $H>0$ with $\limsup_{x\to\infty}\sup_{h\in [-H,H]}\left|\frac{a'(x+h)}{a(x)} \right|<\infty$ (resp. $x\to -\infty$). Then, the quantities $\left|\frac{a(x+\rho)}{a(x)}\right|$ are eventually bounded for all $-H\leq \rho\leq H.$
\end{lemma}

\begin{proof}
Let $\rho\in [-H,H].$ The mean value theorem furnishes a point $h_\rho\in [-H,H]$ with $$\left| \frac{a(x+\rho)}{a(x)}-1 \right| =\left|\frac{a'(x+h_\rho)}{a(x)} \right|$$ from which we have the result by our hypothesis.
\end{proof}

We close this section with a fact about the sets $D_k.$ We show that for a function of $S$ if the $\limsup$ that appears in $D_0$ is a limit, then it has to be equal to $0$. 

\begin{proposition}
If $a\in S$ with $\lim_{x\to\infty}\sup_{h\in [-1,1]}\left|\frac{(a^{-1})'(x+h)}{a^{-1}(x)}\right|=\alpha\in \mathbb{R},$ then $\alpha=0.$
\end{proposition}

\begin{proof}
We assume without loss of generality  that $a>0$ (the case $a<0$ is analogous) and to the contrary that $\alpha>0$ and let $0<\varepsilon<\alpha.$ Using the hypothesis, there exists $M>0$ such that $\frac{(a^{-1})'(x+1)}{a^{-1}(x)}>\alpha-\varepsilon$ for all $x>M,$ or equivalently $$\frac{1}{a^{-1}(x+1)a'(a^{-1}(x+1))}>(\alpha-\varepsilon)\frac{a^{-1}(x)}{a^{-1}(x+1)},$$ from which (using the fact that $1/a'(x)\prec x$ and by taking the $\limsup$) we have that $\lim_{x\to\infty}\frac{a^{-1}(x-1)}{a^{-1}(x)}=0.$ By the mean value theorem, there exists $\xi_x\in[x-1,x]$ with $$\left|\frac{a^{-1}(x-1)}{a^{-1}(x)}-1\right|=\frac{(a^{-1})'(\xi_x)}{a^{-1}(x)}\leq \frac{(a^{-1})'(x+1)}{a^{-1}(x)}.$$ Letting $x\to\infty,$ we get $\alpha\geq 1.$ By Lemma~\ref{L:bound}, there exists $c>0$ such that  $$\frac{(a^{-1})'(x)}{a^{-1}(x)}\geq c\frac{(a^{-1})'(x+1)}{a^{-1}(x)}>c(\alpha-\varepsilon).$$ by integrating and solving for $a(x)$ this relation we get $a(x)\leq c_1\log x +c_2$ for some $c_1,$ $c_2$ constants with $c_1>0.$ Then we have $$0<\frac{1}{c_1}\leq \lim_{x\to\infty}\frac{\log x}{a(x)}=\lim_{x\to\infty}\frac{1}{xa'(x)}=0,$$ a contradiction. The claim now follows.
\end{proof}



   

\section{The key invariant properties}\label{se:3}

\subsection{The sublinear case}

In this subsection we develop the main tool, Lemma~\ref{lem:inv}, in order to prove Theorem~\ref{Thm:LimitMeasure}, which we'll use in the proof of Theorem~\ref{C:main}. 

Recall from $\S$~\ref{se:1} that $\lambda_{\vec{x}}$ is any weak limit in $M(X_1\times \cdots\times X_d)$ of  $$\lambda_{N,\vec{x}}=\frac{1}{N}\sum_{n=0}^{N-1}\Big(T_{1}^{[a_1(n)]}\times\dots\times T_{d}^{[a_d(n)]}\Big)\delta_{\vec{x}}.$$


\begin{lemma}\label{lem:inv}
	Let $d\in \mathbb{N}$ and $a_i\in \mathcal{S},$ $1\leq i\leq d,$ with $a_d\prec\ldots\prec a_1$ and $a'_d\prec \ldots\prec a'_1.$  Then $\lambda_{\vec{x}}$ is invariant under $T_{1}\times id\times\dots\times id$.
\end{lemma}

\begin{proof}
	We can assume without loss of generality that $a_1$ is eventually positive (the other case is analogous). For $\vec{b}=(b_1,b_2,\ldots,b_d)\in \mathbb{Z}^d$, write $\vec{b}_\ast=(b_2,\ldots,b_d)$ and let $$\mathcal{U}_{b_1,\vec{b}_\ast}:=\left|\Big\{n\in \{1,\ldots,N\}:\;b_i\leq a_i(n)<b_i+1\;\forall\;1\leq i\leq d\Big\}\right|,\;\delta_{b_1,\vec{b}_\ast}:=\delta_{T_1^{b_1}x_1}\times \cdots\times \delta_{T_d^{b_d}x_d}.$$
	\noindent With this notation, we have that $\lambda_{N,\vec{x}}=\frac{1}{N}\sum_{(b_1,\vec{b}_\ast):\;\mathcal{U}_{b_1,\vec{b}_\ast}\neq 0}\mathcal{U}_{b_1,\vec{b}_\ast}\delta_{b_1,\vec{b}_\ast},$
	and if $$\tilde{\lambda}_{N,\vec{x}}:=(T_{1}\times id\times\cdots\times id)(\lambda_{N,\vec{x}})=\frac{1}{N}\sum_{(b_1,\vec{b}_\ast):\;\mathcal{U}_{b_1-1,\vec{b}_\ast}\neq 0}\mathcal{U}_{b_1-1,\vec{b}_\ast}\delta_{b_1,\vec{b}_\ast}, $$ 
	then, 
	\begin{equation}\label{E:4'}
	\lambda_{N,\vec{x}}- \tilde{\lambda}_{N,\vec{x}}  =  \frac{1}{N}\sum_{(b_1,\vec{b}_\ast):\;\mathcal{U}_{b_1,\vec{b}_\ast}\neq 0, \; \mathcal{U}_{b_1-1,\vec{b}_\ast}\neq 0}\Big(\mathcal{U}_{b_1,\vec{b}_\ast}-\mathcal{U}_{b_1-1,\vec{b}_\ast}\Big)\delta_{b_1,\vec{b}_\ast}
	\end{equation}
	\begin{equation}\label{E:5'}
	+  \frac{1}{N}\sum_{(b_1,\vec{b}_\ast):\;\mathcal{U}_{b_1,\vec{b}_\ast}\neq 0, \; \mathcal{U}_{b_1-1,\vec{b}_\ast}= 0}\mathcal{U}_{b_1,\vec{b}_\ast}\delta_{b_1,\vec{b}_\ast}\;\;\;
	\end{equation}
	\begin{equation}\label{E:6'}
	\; -  \frac{1}{N}\sum_{(b_1,\vec{b}_\ast):\;\mathcal{U}_{b_1,\vec{b}_\ast}= 0, \; \mathcal{U}_{b_1-1,\vec{b}_\ast}\neq 0}\mathcal{U}_{b_1-1,\vec{b}_\ast}\delta_{b_1,\vec{b}_\ast}.
	\end{equation}
	
	\medskip
	
	\noindent We will study each term separately. Terms ~\eqref{E:5'} and ~\eqref{E:6'} are treated similarly, so we only study the Terms~\eqref{E:4'} and~\eqref{E:5'}.
	
	\medskip
	
	\noindent {\bf{Term~\eqref{E:4'}:}} For $1\leq i\leq d$,  let $$C_{a_i,b_i}:=\min\{a_i^{-1}(b_i),\;a_i^{-1}(b_i+1)\}\;\;\text{and}\;\;C'_{a_i,b_i}:=\max\{a_i^{-1}(b_i),\;a_i^{-1}(b_i+1)\}.$$
	The conditions $\mathcal{U}_{b_1,\vec{b}_\ast}\neq 0$ and  $\mathcal{U}_{b_1-1,\vec{b}_\ast}\neq 0$ imply that 
	\begin{equation}\label{E:4}
	\max_{1\leq i\leq d}\{C_{a_i,b_i}\}  <  \min_{1\leq i\leq d}\{C'_{a_i,b_i}\};\;\;\;\hbox{and}
	\end{equation}
	\begin{equation}\label{E:5}
	\max_{2\leq i\leq d}\{C_{a_1,b_1-1},\;C_{a_i,b_i}\}  <  \min_{2\leq i\leq d}\{C'_{a_1,b_1-1},\;C'_{a_i,b_i}\}.
	\end{equation}
	\noindent Using the Relations \eqref{E:4} and \eqref{E:5} we get $$C_{a_i,b_i}<a_1^{-1}(b_1)<C'_{a_i,b_i}\;\;\hbox{for all}\;\;2\leq i\leq d,$$
	from which we have $$b_i<a_i\circ a_1^{-1}(b_1)<b_i+1,$$ hence $b_i=[a_i\circ a_1^{-1}(b_1)]$ for all $2\leq i\leq d.$
	
	\medskip
	
	\noindent By the mean value theorem, for some $\xi_{b}\in (b_1-1,b_1+1)$ we have
	\begin{equation*}\label{E:6}
	\left|\mathcal{U}_{b_1,\vec{b}_\ast}-\mathcal{U}_{b_1-1,\vec{b}_\ast}\right|\leq \left|a_1^{-1}(b_1+1)-2a_1^{-1}(b_1)+a_1^{-1}(b_1-1)\right|+2= \left| (a_1^{-1})''(\xi_b) \right|+2.
	\end{equation*}
	 Since $a_1^{-1}\in D_2\cup M_2$ for large enough $b_1,$ we have that $\left|(a_1^{-1})''(\xi_{b})\right|\ll \left|(a_1^{-1})''(b_1)\right|$ (or $\left|(a_1^{-1})''(b_1\pm 1)\right|$ which is studied similarly). So, $$\frac{1}{N}\sum_{b_1=1}^{[a_1(N)]-1}\left|\mathcal{U}_{b_1,\vec{b}_\ast}-\mathcal{U}_{b_1-1,\vec{b}_\ast}\right|\ll
	\frac{1}{N}\sum_{b_1=1}^{[a_1(N)]-1}\left| (a_1^{-1})''(b_1) \right|$$  (the term $2[a_1(N)]/N$ is null by sublinearity of $a_1$). By the same property, for large enough $b_1$ and all $t\in [b_1,b_1+1]$, we have that $\left|(a_1^{-1})''(b_1)\right|\ll \left|(a_1^{-1})''(t)\right|,$ so $$\left|(a_1^{-1})''(b_1)\right|\ll \int_{b_1}^{b_1+1}\left|(a_1^{-1})''(t)\right|\;dt.$$
	Hence,
	\begin{equation}\label{E:7}
	\frac{1}{N}\sum_{b_1=1}^{[a_1(N)]-1}\left| (a_1^{-1})''(b_1) \right|\ll \frac{1}{N}\int_\ast^{[a_1(N)]}\left| (a_1^{-1})''(t) \right|\;dt\ll  \frac{\left|(a_1^{-1})'([a_1(N)])\right|}{N}
	\end{equation}
	(where we used the fact that $a_1^{-1}\in D_2$ so $(a_1^{-1})''$ has constant sign -- in our case here is positive).
	Using the fact that $a_1^{-1}\in D_1,$ we have  that $(a_1^{-1})'([a_1(N)])/(a_1^{-1})'(a_1(N))$ is bounded for large $N$ 
	so the right hand side of ~\eqref{E:7} is bounded by a constant multiple of $(Na_1'(N))^{-1}$ which goes to $0$ as $N\to\infty$  and consequently Term~\eqref{E:4'} goes to $0$ as $N\to \infty.$
	
	\medskip
	
	\noindent {\bf{Term~\eqref{E:5'}:}} The conditions $\mathcal{U}_{b_1,\vec{b}_\ast}\neq 0$ and  $\mathcal{U}_{b_1-1,\vec{b}_\ast}= 0$ imply that 
	\begin{equation}\label{E:8}
	\max_{1\leq i\leq d}\{C_{a_i,b_i}\}  <  \min_{1\leq i\leq d}\{C'_{a_i,b_i}\},\;\;\;\hbox{and}
	\end{equation}
	\begin{equation}\label{E:9}
	\hbox{as in \eqref{E:5} but with}\;\; C_{a_i,b_i}  >  a_1^{-1}(b_1)-1\;\;\; \hbox{for at least one} \;\;2\leq i\leq d.
	\end{equation}
	\noindent  If $\{i_0=1,i_1,\ldots,i_r\}$ is the set of indices for which $C_{a_i,b_i}  >  a_1^{-1}(b_1)-1,$ let 
	 $$B_{\{i_0,i_1,\ldots,i_r\}}:=\{\vec{b}_\ast:\;a_1^{-1}(b_1)\leq C_{a_{i_1},b_{i_1}}\leq\ldots\leq C_{a_{i_r},b_{i_r}}\}.$$ 
	Then we have that $$a_1^{-1}(b_1)-1< C_{a_{i_j},b_{i_j}}\leq C_{a_{i_r},b_{i_r}}<C'_{a_{i_j},b_{i_j}}\;\;\hbox{for all}\;\;0< j\leq r,$$ from which we get $$b_{i_j}\leq a_{i_j}\circ a^{-1}_{i_r}(b_{i_r}+e_{i_r})<b_{i_j}+1,$$ for some $e_{i_r}\in \{0,1\}.$ Hence, $b_{i_j}=\left[a_{i_j}\circ a^{-1}_{i_r}(b_{i_r}+e_{i_r})\right]$ for all $0< j\leq r.$
	
	\medskip
	
	For $j=0,$ we have $$a_1^{-1}(b_1)-1<C_{a_{i_r},b_{i_r}}<a_1^{-1}(b_1+1),$$ or $$a_1(C_{a_{i_r},b_{i_r}})-1<b_1<a_1(C_{a_{i_r},b_{i_r}}+1).$$ So, for each $b_{i_r}$ fixed, $b_1=[a_1(C_{a_{i_r},b_{i_r}})]$ or $b_1=[a_1(C_{a_{i_r},b_{i_r}}+1)].$
	
	\medskip
	
	\noindent For $i\notin \{i_1,\ldots,i_r\}$, we have $$C_{a_i,b_i}\leq a_1^{-1}(b_1)< C_{a_{i_r},b_{i_r}}<C'_{a_i,b_i},$$ or $$b_{i}\leq a_{i}\circ a^{-1}_{i_r}(b_{i_r}+e_{i_r})<b_{i}+1,$$ for some $e_{i_r}\in \{0,1\},$ hence $b_{i}=\left[a_{i}\circ a^{-1}_{i_r}(b_{i_r}+e_{i_r})\right].$
	
	\medskip
	
	\noindent  The average of $\left|\mathcal{U}_{b_1,\vec{b}_\ast}\right|$ can be split into finitely many sums over $(b_1,\vec{b}_\ast)$ for $\vec{b}_\ast\in B_{\{i_0,i_1,\ldots,i_r\}}$ for some $\{i_0,i_1,\ldots,i_r\}$. For $\vec{b}_\ast\in B_{\{i_0,i_1,\ldots,i_r\}},$ $\left|\mathcal{U}_{b_1,\vec{b}_\ast}\right|$ is bounded by $$\left|a_1^{-1}(b_1+1)-a_1^{-1}(b_1)\right|+1\ll\left|(a_1^{-1})'(b_1)\right|+1= \left|(a_1^{-1})'([a_1\circ a^{-1}_{i_r}(b_{i_r}+e_{i_r})])\right|+1.$$ This approximation follows  by the mean value theorem and the fact that $a_1^{-1}\in D_1.$
	Hence, every average is estimated by a constant multiple of
	\begin{equation}\label{E:10}
	\frac{1}{N}\left|\int_{\ast}^{[a_{i_r}(N)]} \left|(a_1^{-1})'([a_1\circ a^{-1}_{i_r}(t)])\right|\;dt\right|\ll \frac{1}{N}\left|\int_{\ast}^{[a_{i_r}(N)]} \left|(a_1^{-1})'(a_1\circ a^{-1}_{i_r}(t))\right|\;dt\right|, 
	\end{equation} where we used the fact that $a_{i_r}\prec a_1$ and the sublinearity of $a_{i_r}.$ So, the right hand side of \eqref{E:10} is bounded by $$\frac{1}{N}\left|\int_\ast^{[a_{i_r}(N)]}\frac{1}{\left|a_1'(a_{i_r}^{-1}(t))\right|}\;dt\right|=\frac{1}{N}\left|\int_\ast^{a_{i_r}^{-1}([a_{i_r}(N)])}\left|\frac{a_{i_r}'(t)}{a_1'(t)}\right|\;dt\right|.$$ Since $a^{-1}_{i_r}\in D_0$ we have $a_{i_r}^{-1}([a_{i_r}(N)])/N\to 1$ as $N\to\infty,$ hence the last integral goes to 0, since the integrand is arbitrarily small for large $N$ because of $a'_{i_r}\prec a'_1.$
	
	\noindent  This proves that the Terms~\eqref{E:5'} and~\eqref{E:6'} go to $0$ as $N\to \infty$ and we get the conclusion.
\end{proof}



\subsection{Linearity of the fastest term}

In this subsection we deal with the case where the fastest function is linear. If the leading coefficient is rational, we can work as in Lemma~\ref{lem:inv} while if the leading coefficient is irrational, with an explicit example (see the remark after Lemma~\ref{lem:inv1}) we show that the invariant property of Lemma~\ref{lem:inv} can fail, but one can still deal with this case via the suspension flow and Birkhoff's ergodic theorem (see $\S$~\ref{se:4.1}).

\subsubsection{{\bf{Rational leading coefficient}}}
 If $a_1(n)=kn+\ell,$ with $k\in \mathbb{Q}\setminus\{0\},$ $\ell\in \mathbb{R},$ it suffices to cover the case where $a_1(n)=pn$ with $p\in \mathbb{Z}\setminus\{0\}$ (see details in the proof of Theorem~\ref{Thm:Linear}). As before, let $\lambda_{\vec{x}}$ be a weak limit of $\lambda_{N,\vec{x}}$. 





\begin{lemma}\label{lem:inv1}
For $d\in \mathbb{N},$ let $a_1(n)=pn$, $p\in \mathbb{Z}\setminus\{0\}$ and $a_i,$ $2\leq i\leq d$  be functions with $a_{d}\prec\ldots\prec a_1.$ Then $\lambda_{\vec{x}}$ is invariant under $T^{p}_{1}\times id\times\dots\times id$.
\end{lemma}

\begin{proof}
We can assume that $p>0$ (the case where $p<0$ is analogous), hence $a_1$ is eventually positive. 
  For $\vec{b}=(b_1,b_2,\ldots,b_d)\in \mathbb{Z}^d$ we write $\vec{b}_\ast=(b_2,\ldots,b_d)$ and we  set $$ \mathcal{U}_{b_1,\vec{b}_\ast}=\left\{ \begin{array}{ll} 1 \; ,\text{ if } [a_{i}(b_{1})]=b_{i}\; \forall\; 2\leq i\leq d \\ 0 \; ,\text{ otherwise } \end{array} \right. ,\;\text{and}\;
  \delta_{b_1,\vec{b}_\ast}=\delta_{T_1^{[a_{1}(b_1)]}x_1}\times\delta_{T_2^{b_{2}}x_2}\times \cdots\times \delta_{T_d^{b_d}x_d}$$ (note that the position of $b_{1}$ here is not the same as in the proof of Lemma~\ref{lem:inv}). 
\noindent With this notation, we have that $\lambda_{N,\vec{x}}=\frac{1}{N}\sum_{(b_1,\vec{b}_\ast):\;\mathcal{U}_{b_1,\vec{b}_\ast}\neq 0}\mathcal{U}_{b_1,\vec{b}_\ast}\delta_{b_1,\vec{b}_\ast},$
 and $$\tilde{\lambda}_{N,\vec{x}}:=(T^{p}_{1}\times id\times\cdots\times id)(\lambda_{N,x})=\frac{1}{N}\sum_{(b_1,\vec{b}_\ast):\;\mathcal{U}_{b_1-1,\vec{b}_\ast}\neq 0}\mathcal{U}_{b_1-1,\vec{b}_\ast}\delta_{b_1,\vec{b}_\ast}, $$ 
 where we used the fact that $[a_{1}(b_{1}-1)]+p=[a_{1}(b_{1})].$
Then, 
\begin{equation}\label{E:13'}
\lambda_{N,\vec{x}}- \tilde{\lambda}_{N,\vec{x}}  =  \frac{1}{N}\sum_{(b_1,\vec{b}_\ast):\;\mathcal{U}_{b_1,\vec{b}_\ast}\neq 0, \; \mathcal{U}_{b_1-1,\vec{b}_\ast}\neq 0}\Big(\mathcal{U}_{b_1,\vec{b}_\ast}-\mathcal{U}_{b_1-1,\vec{b}_\ast}\Big)\delta_{b_1,\vec{b}_\ast}
\end{equation}
\begin{equation}\label{E:14'}
 +  \frac{1}{N}\sum_{(b_1,\vec{b}_\ast):\;\mathcal{U}_{b_1,\vec{b}_\ast}\neq 0, \; \mathcal{U}_{b_1-1,\vec{b}_\ast}= 0}\mathcal{U}_{b_1,\vec{b}_\ast}\delta_{b_1,\vec{b}_\ast}\;\;
\end{equation}
\begin{equation}\label{E:15'}
\;\; -  \frac{1}{N}\sum_{(b_1,\vec{b}_\ast):\;\mathcal{U}_{b_1,\vec{b}_\ast}= 0, \; \mathcal{U}_{b_1-1,\vec{b}_\ast}\neq 0}\mathcal{U}_{b_1-1,\vec{b}_\ast}\delta_{b_1,\vec{b}_\ast}.
\end{equation}
\noindent We study the Terms~\eqref{E:13'} and~\eqref{E:14'} as in the sublinear case.

\medskip

%
%
\noindent {\bf{Term~\eqref{E:13'}:}}
Since each $\mathcal{U}_{b_{1},b_{\ast}}$ is at most 1,
 $\mathcal{U}_{b_{1},b_{\ast}}=\mathcal{U}_{b_{1}-1,b_{\ast}}=1,$ so
Term~\eqref{E:13'} equals to 0.

\medskip

\noindent {\bf{Term~\eqref{E:14'}:}} 
 Let $$C_{a_i,b_i}:=\min\{a_i^{-1}(b_i),\;a_i^{-1}(b_i+1)\}\;\;\text{and}\;\;C'_{a_i,b_i}:=\max\{a_i^{-1}(b_i),\;a_i^{-1}(b_i+1)\}.$$ 
The conditions $\mathcal{U}_{b_1,\vec{b}_\ast}\neq 0$ and  $\mathcal{U}_{b_1-1,\vec{b}_\ast}= 0$ imply that 
\begin{equation}\label{E:81}
\max_{1\leq i\leq d}\{C_{a_i,b_i}\} \leq b_{1}  <  \min_{1\leq i\leq d}\{C'_{a_i,b_i}\},\;\;\;\hbox{and}
\end{equation}
\begin{equation}\label{E:91}
C_{a_i,b_i}  \geq  b_{1}-1\;\;\; \hbox{for at least one} \;\;2\leq i\leq d.
\end{equation}
\noindent  If $\{i_1,\ldots,i_r\}$ is the set of indices for which $C_{a_i,b_i}  \geq  b_{1}-1,$ let 
$$B_{\{i_1,\ldots,i_r\}}:=\{\vec{b}_\ast:\;b_1-1\leq C_{a_{i_1},b_{i_1}}\leq\ldots\leq C_{a_{i_r},b_{i_r}}\}.$$ 
 As in Lemma~\ref{lem:inv} we get $b_1=[C_{a_{i_r},b_{i_r}}]$ or $b_1=[C_{a_{i_r},b_{i_r}}]+1$ and $b_i=\left[a_{i}\circ a^{-1}_{i_r}(b_{i_r}+e_{i_r})\right],$ $2\leq i\leq d$ for some $e_{i_r}\in \{0,1\}.$

  For fixed $b_{i_r}$  there are at most $2^{r}$ terms $\left|\mathcal{U}_{b_1,\vec{b}_\ast}\right|$ in \eqref{E:14'} whose last index equals to $b_{i_{r}}$. Since $\vert b_{i_{r}}\vert\leq a_{i_{r}}(N)$, \eqref{E:14'} is bounded by a constant multiple of
 $a_{i_{r}}(N)/N$, which goes to $0$ as $N\to \infty$ and hence the conclusion of the lemma follows.
\end{proof} 

\begin{remark*}
 In the previous proof we essentially used the relation $[a_{1}(b_{1}-1)]+p=[a_{1}(b_{1})],$ which is not in general true for expressions of the form $a(t)=\gamma n+\ell,$ $\gamma\notin \mathbb{Q}.$ Actually, Lemma~\ref{lem:inv1} cannot even be extended to the $d=1$ case if $a_{1}(n)=\gamma n+\ell,\gamma\notin\mathbb{Q}.$

Indeed, let $d=1$ and $(\mathbb{T},\mu)$ be the 1 dimensional torus endowed with the Lebesgue measure. Let $T\colon\mathbb{T}\to\mathbb{T}$ with $Tx=\gamma^{-1}+x$. 
Then ($\{\cdot\}$ denotes the fractional part)
$$T^{[\gamma n+\ell]}x=\gamma^{-1}[\gamma n+\ell]+x=\gamma^{-1}\ell+x-\gamma^{-1}\{\gamma n\}.$$
Since $(\{\gamma n\})_{n}$ is equidistributed on $\mathbb{T}$, for all $f\in L^{\infty}(\mu)$, we have
$$\lim_{N\to\infty}\frac{1}{N}\sum_{n=0}^{N-1}T^{[\gamma n+\ell]}f(x)=\lim_{N\to\infty}\frac{1}{N}\sum_{n=0}^{N-1}f(\gamma^{-1}\ell+x-\gamma^{-1}\{\gamma n\})=\int_{\gamma^{-1}(\ell-1)+x}^{\gamma^{-1}\ell+x}f(y)\, d\mu(y)$$
and
$$\lim_{N\to\infty}\frac{1}{N}\sum_{n=0}^{N-1}T^{[\gamma n+\ell]+1}f(x)=\int_{\gamma^{-1}+x}^{\gamma^{-1}(\ell+1)+x}f(y)\, d\mu(y).$$
So $\lambda_{x}$ is obviously not $T$-invariant.  Nevertheless, we will show in the next section that there is an explicit expression for the limit $\lim_{N\to\infty}\frac{1}{N}\sum_{n=0}^{N-1}T^{[\gamma n+\ell]}f(x), \gamma\notin\mathbb{Q}$. 
 \end{remark*}




\section{Single pointwise averages}\label{se:4}

In this section, we establish all the pointwise limits of (\ref{E:central}) for $d=1$ for the class $\mathcal{T}\cup\mathcal{L}_\ast,$ where $\mathcal{L}_\ast$ denotes the set of linear functions with non-zero leading coefficient. 

\begin{theorem}\label{thm:BKQW2} 
	Let $(X,\mu,T)$ be a measure preserving system and $f\in L^{\infty}(\mu)$. If $a\in \mathcal{T},$ then for $\mu$-a.e. $x\in X$
	\begin{equation}\nonumber
	\lim_{N\to\infty}\frac{1}{N}\sum_{n=0}^{N-1}T^{[a(n)]}f(x)=\mathbb{E}(f\vert\mathcal{I}(T))(x).
	\end{equation}
\end{theorem}	

 Theorem~\ref{thm:BKQW2} is proved in \cite[Theorem~3.4]{BKQW} for $a\in\mathcal{H}^s_\varepsilon$ but the same result holds for $a\in\mathcal{T}$. For completeness we sketch the idea of the proof.


 By \cite[Theorem~3.5]{BK}, if $a$ satisfies some properties, then  the average of a sequence $(x_n)_n$ along $([a(n)])_n$, i.e., $\frac{1}{N}\sum x_{[a(n)]},$ converges to $0$ as long as the usual average $\frac{1}{N}\sum x_{n}$  converges to $0.$ The proof that the assumptions of \cite[Theorem~3.5]{BK} are satisfied for functions from $\mathcal{T}$ can be found in \cite[Lemma~2.5]{BK}. 
 	 More specifically, \cite[Lemma~2.5]{BK} covers the case where $a>0$ while the case where $a<0$ follows by the fact that $[a(n)]=-[-a(n)]-1$ in a set of density 1 since $[a(n)]=-[-a(n)]$ only happens when $a(n)$ is an integer,  which up to time $N$,  happens at most $a(N)$ times. The claim now follows by the sublinearity of $a.$  
 	 
\begin{remark*}
 Combining this result with Proposition~\ref{P:psub}, we get that $\mathcal{T}$ is a subset of $\mathcal{S}^{\ast}$. 
 In fact,  $\mathcal{S}^{\ast}$ properly contains $\mathcal{T}.$  Indeed, 
recall by $\S$~\ref{se:2} that for a small positive $\alpha$ we have that $$a_2(x)=x^\alpha(4/\alpha+\sin\log x)^3\in \mathcal{S}\setminus \mathcal{T}.$$ 
The function $a_{2}$ belongs to $\mathcal{S}^{\ast}$ as it satisfies all the conditions of \cite[Theorem~3.5]{BK}.

To be more precise,
let $\phi(n)=|\{m\in\mathbb{N}:\;n=[a_2(m)]\}|$ and $\Phi(n)=\sum_{k=0}^n\phi(k)$. Then all the following conditions of \cite[Theorem~3.5]{BK} hold

\medskip

\noindent 	(i) $\lim_{n\to\infty}\phi(n)=\infty$ (obvious);

\medskip
	
\noindent 	(ii) $\phi(n)$ is almost increasing, i.e., $\phi(n)=q_n+p_n,$ where $(q_n)_n$ is increasing and $(p_n)_n$ is bounded (this follows by \cite[Lemma~2.4]{BK} since $a_2^{-1}$ is increasing and $a_2'$ is decreasing);

\medskip
	
\noindent (iii) $n\phi(n)/\Phi(n)\leq c$ for some $c>0$ (this follows by the proof of \cite[Lemma~2.5]{BK} since $\frac{n\phi(n)}{\Phi(n)}$ is arbitrarily close to $\frac{n(a_2^{-1})'(n)}{a_2^{-1}(n)},$ hence it's bounded above by a positive number since, as we already mentioned in $\S$~\ref{se:2},  $\frac{na_2'(n)}{a_2(n)}$ is bounded below by $\alpha-\frac{3\alpha}{4-\alpha}>0$).

  Also, it is worth recalling at this point that Theorem~\ref{D:sUniqueErgodic} establishes the same result as Theorem~\ref{thm:BKQW2} for a larger class of functions in the case of uniquely ergodic systems. 
\end{remark*}


 We can now show that Theorem~\ref{Thm:LimitMeasure} holds for the class $\mathcal{T}$.

\begin{corollary} \label{cor:1}
	Let $d\in \mathbb{N},$ $(X_i,\mu_i,T_i),$  be measure preserving systems, $a_i\in \mathcal{T},$ $1\leq i\leq d$ with $a_d\prec \ldots\prec a_1$ and $\nu$ be any coupling of the spaces $(X_i,\mu_i)$. Then, for $\nu$-a.e. $\vec{x} \in X_1\times \cdots \times X_d$,   we have that   $\lambda_{N,\vec{x}}$ converges to $\mu_{[T_1,\ldots,T_d],\vec{x}}$ as $N\to\infty.$ 
\end{corollary}

   This corollary follows (together with the analogous results of Theorems~\ref{C:main}, ~\ref{Thm:Linear}, ~\ref{D:sUniqueErgodic} and ~\ref{D:mainUniqueErgodic}   with the condition $a_i\in \mathcal{S}^{\ast},$ or $\mathcal{S},$ $a_d\prec \ldots\prec a_1$ and $ a'_d\prec \ldots\prec a'_1$ replaced by $a_i\in \mathcal{T},$ or $\mathcal{T}\cup\mathcal{L}_{\ast}$ in Theorem ~\ref{Thm:Linear},  and $a_d\prec \ldots\prec a_1$) by the fact that $\mathcal{T}\subseteq \mathcal{S}^\ast$ and that if $a_1,a_{2}\in \mathcal{T}\cup \mathcal{L}_\ast$ with $a_2\prec a_1$ (so $a_2\in\mathcal{T}$) then we have that $a_2'\prec a_1'$ by the identity $$\frac{a_2'(x)}{a_1'(x)}=\frac{xa_2'(x)}{a_2(x)}\cdot\frac{a_1(x)}{xa_1'(x)}\cdot\frac{a_2(x)}{a_1(x)}.$$
	Of course, for functions of different growth $a_1\prec a_2$ in $S^\ast$ in general, if some of the limits $xa_i'(x)/a_i(x)$ doesn't exist as $x\to\infty$ (as in the case of the function that we saw in  the remark of this section) we don't expect to get the different growth rate of the derivatives, so we have to assume it.

\medskip

 If for every $\alpha\in (0,1]$ we set $\mathcal{T}(\alpha):=\left\{a\in \mathcal{T}:\;\lim_{x\to\infty}\frac{xa'(x)}{a(x)}=\alpha\right\}$ then the following remark gives us a relation between the growth rate of $a_i\in \mathcal{T}(\alpha_i),$ $i=1,2,$ and $\alpha_1,$ $ \alpha_2.$

\begin{remark*}
If $a_i\in \mathcal{T}(\alpha_i),$ $i=1,2,$ with $a_2\prec a_1,$ then $\alpha_2\leq \alpha_1.$ 

Indeed, this follows by \cite[Lemmas 2.1 and 2.6]{BK}. Let $a=a_2/a_1$ and note that $\lim_{x\to\infty}\frac{xa'(x)}{a(x)}=\alpha_2-\alpha_1.$ If $\alpha_2>\alpha_1,$ then by the argument in \cite[Lemma 2.1]{BK}, we have that $|a(x)|\to\infty $ as $x\to\infty,$ a contradiction (note that we only used here the fact that $a$ is bounded).

\noindent Conversely, if $\alpha_2<\alpha_1$ then $a_2\prec a_1.$

 The same argument, for $1/a,$ gives us that $a(x)\to 0$ as $x\to\infty,$ hence $a_2\prec a_1.$

\noindent Note at this point that it can happen $a_2\prec a_1$ while both functions belong to the same $\mathcal{T}(\alpha),$ as  $a_2(x)=x^{1/3}$ and $a_1(x)=x^{1/3}\log x,$ where $a_2\prec a_1$ with $a_1, a_2\in \mathcal{T}(1/3).$
\end{remark*}

\medskip

We now deal with the irrational leading coefficient case in the class $\mathcal{L}_\ast.$

\subsection{{\bf{Irrational leading coefficient}}}\label{se:4.1}
In this subsection, we study the limit along linear iterates of the form $\lim_{N\to\infty}\frac{1}{N}\sum_{n=0}^{N-1}T^{[\gamma n+\ell]}f(x)$ when $\gamma$ is irrational. To do this we have to introduce some additional notions and tools.

\subsubsection{A generalized ergodic theorem} Let $(X,\mu,T)$ be a measure preserving system, $\mathcal{I}_{\gamma,m}(T)$ be the sub-$\sigma$-algebra generated by the eigenspace of $T$ with eigenvalue $-\frac{m}{\gamma}$, $\mathcal{U}_{\gamma,m} \subseteq L^2(\mu)$ be the  closed subalgebra generated by all functions of the form $Tg-\exp\Big(2\pi i \frac{m}{\gamma}\Big)g,$ $\mathcal{I}_{\gamma}(T)=\bigvee_{m\in\mathbb{Z}}\mathcal{I}_{\gamma,m}(T)$ and $\mathcal{U}_{\gamma}(T)=\bigcap_{m\in\mathbb{Z}}\mathcal{U}_{\gamma,m}(T)$. We have the following structure theorem (its proof is routine and we omit it):

\begin{theorem}\label{Thm:str} Let $(X,\mu,T)$ be a measure preserving system and $\gamma\in\mathbb{R}\setminus \mathbb{Q}$. Then 
	$$L^{2}(\mu)=\mathcal{I}_{\gamma,m}(T)\oplus\mathcal{U}_{\gamma,m}(T)$$
	for all $m\in\mathbb{Z}$. In particular, 
	$$L^{2}(\mu)=\mathcal{I}_{\gamma}(T)\oplus\mathcal{U}_{\gamma}(T).$$
\end{theorem}

We also have the following von Neumann-type mean ergodic theorem:

\begin{proposition}\label{Prop:irequ}
	Let $(X,\mu,T)$ be a measure preserving system, $\gamma\in\mathbb{R}\backslash\mathbb{Q}$, $\ell\in\mathbb{R}$ and $f\in L^{2}(\mu).$ We have the following (each convergence takes place in $L^2(\mu)$):
	\begin{itemize}
		\item If $\mathbb{E}(f\vert\mathcal{I}_{\gamma}(T))=0$, then $$\lim_{N\to\infty}\frac{1}{N}\sum_{n=0}^{N-1}T^{[\gamma n+\ell]}f=0;$$
		\item If $f$ is measurable in $\mathcal{I}_{\gamma,m}(T)$, then $$\lim_{N\to\infty}\frac{1}{N}\sum_{n=0}^{N-1}T^{[\gamma n+\ell]}f=\exp\Big(2\pi i\frac{m\ell}{\gamma}\Big)\cdot\frac{\exp(-2\pi i\frac{m}{\gamma})-1}{-2\pi i \frac{m}{\gamma}}f.$$
	\end{itemize}	
	Consequently, for all $f\in L^{2}(\mu)$, we have 
	$$\lim_{N\to\infty}\frac{1}{N}\sum_{n=0}^{N-1}T^{[\gamma n+\ell]}f=\sum_{m\in\mathbb{Z}}\exp\Big(2\pi i\frac{m\ell}{\gamma}\Big)\cdot\frac{\exp(-2\pi i\frac{m}{\gamma})-1}{-2\pi i \frac{m}{\gamma}}\mathbb{E}(f\vert I_{\gamma,m}(T)).$$
\end{proposition}	

\begin{proof}
	Assume that $\mathbb{E}(f\vert\mathcal{I}_{\gamma}(T))=0$.  By Theorem \ref{Thm:str}, $f\in \mathcal{U}_{\gamma}(T)$.
	Let $\varepsilon>0$. By definition,  for all $j\in\mathbb{Z}$, there exist $g_{j}, \varepsilon_j\in L^{2}(\mu)$  such that $f=Tg_{j}-\exp\Big(2\pi i\frac{j}{\gamma}\Big)g_{j}+\varepsilon_j$ and $\|\varepsilon_j\|_2\leq \varepsilon$.
	We may assume that $\gamma>0$ since the proof of the other case is identical.
	We first assume that $\gamma>1$. Note that $m=[n\gamma+\ell]$ for some $n\in\mathbb{Z}$ if and only if $\Big\{\frac{m-\ell}{\gamma}\Big\}\in \Big(1-\frac{1}{\gamma},1\Big)$. Additionally for each $m\in\mathbb{Z}$, there is at most one $n\in\mathbb{Z}$ such that $m=[n\gamma+\ell]$. So we have
	$$\lim_{N\to\infty}\frac{1}{N}\sum_{n=0}^{N-1}T^{[\gamma n+\ell]}f=\lim_{N\to\infty}\frac{\gamma}{N}\sum_{n=0}^{N-1}\textbf{1}_{(1-\frac{1}{\gamma},1)}\Big(\Big\{\frac{n-\ell}{\gamma}\Big\}\Big)T^{n}f.$$
	Let $$\textbf{1}_{(1-\frac{1}{\gamma},1)}\Big(\Big\{x-\frac{\ell}{\gamma}\Big\}\Big)=\sum_{j\in\mathbb{Z}}a_{j}\exp(-2\pi i jx),\;\; a_{j}\in\mathbb{R}$$
	be the Fourier expansion of the function $\textbf{1}_{(1-\frac{1}{\gamma},1)}\Big(\Big\{\cdot-\frac{\ell}{\gamma}\Big\}\Big).$
	Then 
		\begin{equation}\nonumber
		\begin{split}
		&\lim_{N\to\infty}\frac{1}{N}\sum_{n=0}^{N-1}T^{[\gamma n+\ell]}f
		=\lim_{N\to\infty}\frac{\gamma}{N}\sum_{n=0}^{N-1}\Big(\sum_{j\in\mathbb{Z}}a_{j}\exp\Big(-2\pi i \frac{jn}{\gamma}\Big)T^{n}f\Big)
		\\&=\sum_{j\in\mathbb{Z}}a_{j}\left(\lim_{N\to\infty}\frac{\gamma}{N}\Big(\sum_{n=1}^{N}\exp\Big(-2\pi i \frac{j(n-1)}{\gamma}\Big)T^{n}g_{j}-\sum_{n=0}^{N-1}\exp\Big(-2\pi i \frac{j(n-1)}{\gamma}\Big)T^{n}g_{j}\Big)\right)\\
		&\quad + \sum_{j\in\mathbb{Z}}a_{j}\Bigl(\lim_{N\to\infty}\frac{\gamma}{N}\Big(\sum_{n=1}^{N}\exp\Big(-2\pi i \frac{j(n-1)}{\gamma}\Big)T^{n}\varepsilon_j \Big)\Bigr).  
		\end{split}
		\end{equation} 
The first series is equal to 0, while the second one has $L^2$-norm smaller than $\varepsilon$. Since $\varepsilon$ is arbitrary	we conclude that $\lim_{N\to\infty}\frac{1}{N}\sum_{n=0}^{N-1}T^{[\gamma n+\ell]}f=0$. 	
		
	If $0<\gamma<1$, pick $W\in\mathbb{N}$ such that $\gamma W>1$. Then by the previous computation, 	
	$$\lim_{N\to\infty}\frac{1}{N}\sum_{n=0}^{N-1}T^{[\gamma n+\ell]}f
	=\frac{1}{W}\sum_{k=0}^{W-1}\lim_{N\to\infty}\frac{1}{N}\sum_{n=0}^{N-1}T^{[(\gamma W) n+(\gamma k+\ell)]}f=0.$$ 
	This finishes the proof of the first part. Suppose now that $f$ is measurable in $\mathcal{I}_{\gamma,m}(T)$. Then $Tf(x)=\exp\Big(2\pi i \frac{m}{\gamma}\Big)f(x)$ for $\mu$-a.e. $x\in X$. For such $x\in X$, we have
	\begin{equation}\nonumber
	   \begin{split}
	     \frac{1}{N}\sum_{n=0}^{N-1}T^{[\gamma n+\ell]}f(x)
	     & =\frac{1}{N}\sum_{n=0}^{N-1}\exp\Big(2\pi i \frac{m[n\gamma +\ell]}{\gamma}\Big)f(x)
	     \\&=\exp\Big(2\pi i \frac{m\ell}{\gamma}\Big)f(x)\cdot\frac{1}{N}\sum_{n=0}^{N-1}\exp\Big(-2\pi i\frac{m\{n\gamma +\ell\}}{\gamma}\Big).
	   \end{split}
	\end{equation} 
	Since the sequence $(\{n\gamma+\ell\})_{n}$ is equidistributed in $\mathbb{T},$ we have that 
	$$\lim_{N\to\infty}\frac{1}{N}\sum_{n=0}^{N-1}\exp\Big(-2\pi i\frac{m\{n\gamma +\ell\}}{\gamma}\Big)=\int_{\mathbb{T}}\exp\Big(-2\pi i\frac{mx}{\gamma}\Big)\, dx=\frac{\exp(-2\pi i\frac{m}{\gamma})-1}{-2\pi i \frac{m}{\gamma}}.$$
	This completes the proof.
\end{proof}

\subsubsection{A special extension} Let $(x,\mu,T)$ be a measure preserving system.
We build an extension system of $X$ which will be used to prove Theorem~\ref{Thm:Linear}.
Let $([0,1],\mathcal{D},m)$ be the $[0,1]$ interval with the Lebesgue measure $m$. Let $R$ be the equivalence relation on $[0,1]\times X$ generated by $((1,x),(0,Tx)),$ $x\in X,$ $Y=([0,1]\times X)/R,$ $ \mathcal{Y}=(\mathcal{D}\times\mathcal{B})/R,$ $ \nu=m\times\mu/R,$ and $ \tilde{T}=id\times T.$
Then we have a factor map
$\pi\colon(Y,\mathcal{Y},\nu,\tilde{T})\to(X,\mathcal{B},\mu,T)$, where $\pi$ is the projection to the second coordinate. 
Let $S\colon Y\to Y$ with $S(t,x)=S(t+\gamma,x).$ Then (using the relations $[\gamma+\{\gamma'\}]+[\gamma']=[\gamma+\gamma']$ and $\{\gamma+\{\gamma'\}\}=\{\gamma+\gamma'\}$) we have $$S^{n}\Big(\{\ell\},T^{[\ell]}x\Big)=(n\gamma+\ell,x)=\Big(\{n\gamma+\ell\},T^{[n\gamma+\ell]}x\Big)$$ for all $n\in\mathbb{Z}$ and $x\in X.$ For a function $f$ on $X$ we define its extension, $\tilde{f}$ on $Y,$ by $\tilde{f}(t,x)=f(x)$ for all $t\in[0,1), x\in X$. Then $$T^{[\gamma n+\ell]}f(x)=\tilde{f}\Big(\{n\gamma+\ell\},T^{[n\gamma+\ell]}x\Big)=S^{n}\tilde{f}\Big(\{\ell\},T^{[\ell]}x\Big).$$

\begin{corollary}\label{Cor:irequ}
	Let the quantifiers be as above. Then for $\mu$-a.e. $x\in X$ we have $$\mathbb{E}\Big(\tilde{f}\vert\mathcal{I}(S)\Big)\Big(\{\ell\},T^{[\ell]}x\Big)=\sum_{m\in\mathbb{Z}}\exp\Big(2\pi i\frac{m\ell}{\gamma}\Big)\cdot\frac{\exp(-2\pi i\frac{m}{\gamma})-1}{-2\pi i \frac{m}{\gamma}}\mathbb{E}(f\vert I_{\gamma,m}(T))(x).$$
\end{corollary}	

\begin{proof} 
	By Birkhoff's ergodic theorem, the limit $\lim_{N\to\infty}\frac{1}{N}\sum_{n=0}^{N-1}S^{n}\tilde{f}(y,z)$ exists and equals to $\mathbb{E}\Big(\tilde{f}\vert\mathcal{I}(S)\Big)(y,z)$ for $\nu$-a.e. $(y,z)\in Y$. Since $\tilde{f}(y,z)=\tilde{f}(y',z)$ for all $y,y'\in[0,1],$ $ z\in X$, and $\mu$ is $T$-invariant, it is easy to conclude that $\lim_{N\to\infty}\frac{1}{N}\sum_{n=0}^{N-1}S^{n}\tilde{f}\Big(\{\ell\},T^{[\ell]}x\Big)$ exists and equals to $\mathbb{E}\Big(\tilde{f}\vert\mathcal{I}(S)\Big)\Big(\{\ell\},T^{[\ell]}x\Big)$ for $\mu$-a.e. $x\in X$. This implies that the limit
	$$\lim_{N\to\infty}\frac{1}{N}\sum_{n=0}^{N-1}T^{[\gamma n+\ell]}f(x)$$ exists for $\mu$-a.e. $x\in X$, hence its pointwise limit equals to its $L^{2}(\mu)$ limit. So	$$\mathbb{E}\Big(\tilde{f}\vert\mathcal{I}(S)\Big)\Big(\{\ell\},T^{[\ell]}x\Big)=\lim_{N\to\infty}\frac{1}{N}\sum_{n=0}^{N-1}S^{n}\tilde{f}\Big(\{\ell\},T^{[\ell]}x\Big)=\lim_{N\to\infty}\frac{1}{N}\sum_{n=0}^{N-1}T^{[\gamma n+\ell]}f(x)$$
	for $\mu$-a.e. $x\in X.$ The result now follows by Proposition \ref{Prop:irequ}.
\end{proof}

\section{Proof of main results} \label{Sec:MainResults}

In this last section we give the proof of the statements in $\S$~\ref{se:1}.  To lighten the notation, we omit writing the spaces along which we integrate, since they are easily deduced by the measures that we use.

\begin{proof}[Proof of Theorem~\ref{Thm:LimitMeasure}]



	The case $d=1$ is true by the definition of $\mathcal{S}^\ast.$ 	 Now suppose the conclusion holds for $d-1$. Let  $\nu$ be a coupling of the systems $(X_i,\mu_i)$ and let $\pi_{1}\colon X_1\times\cdots\times X_{d}\to X_1$ be the projection to the first coordinate and $\pi_{2}\colon X_1\times\cdots\times X_d\to X_2\times\cdots\times X_d$ be the projection to the rest coordinates. Write $\vec{x}=(x_1,x_{\ast})$, where $x_{\ast}=(x_2,\ldots,x_d)$. By induction hypothesis,  $(\pi_{1})_{*}\lambda_{\vec{x}}=\mu_{[T_{1}],x_1}$ $\nu$-a.e.  and $(\pi_{2})_{*}\lambda_{\vec{x}}=\mu_{[T_{2},\dots,T_{d}],x_{\ast}}$ $\nu$-a.e. 
	
	
	For every $1\leq i\leq d$, fix a countable dense set of continuous functions $\mathcal{C}_i=\{g_{i,k}: k\in \mathbb{N} \}\subseteq C(X_i)$. Let $X_1'\subseteq X_1$ be a full $\mu_1$-measure set such that   $\left(\frac{1}{N}\sum_{n=0}^{N-1}T_{1}^{n}g_{1,k}(x_1)\right)_N$ converges to $\mathbb{E}(g_{1,k}\vert \mathcal{I}({T_1}))(x_1)=\int g_{1,k} \; d\mu_{[T_1],x_1}$ as $N\to\infty$ for every $k\in \mathbb{N}$ and $x_1\in X'_1$. Since $\mu_1=\int \mu_{[T_1],x_1}\; d\mu_1(x_1)$, the same is true for $\mu_{[T_1],x_1}$-a.e. $y \in X_1$ for $\mu_1$-a.e. $x_1\in X_1$. 
	Let $f_i \in \mathcal{C}_i$, $1\leq i\leq d$ and $\vec{x}\in \pi_1^{-1}(X_1')$. By Lemma \ref{lem:inv}  and since $(\pi_{1})_{*}\lambda_{\vec{x}}=\mu_{[T_{1}],x_1}$ we have
	\begin{eqnarray*}
		\int f_{1}\otimes f_{2}\otimes\dots\otimes f_{d}\; d\lambda_{\vec{x}}
		& = & \lim_{N\to\infty}\frac{1}{N}\sum_{n=0}^{N-1}\int T_{1}^{n}f_{1}\otimes f_{2}\otimes\dots\otimes f_{d}\; d\lambda_{\vec{x}}
		\\& = & \int\Bigl(\lim_{N\to\infty}\frac{1}{N}\sum_{n=0}^{N-1}T_{1}^{n}f_{1}\Bigr)\otimes f_{2}\otimes\dots\otimes f_{d}\; d\lambda_{\vec{x}}
		\\& = & \int \mathbb{E}(f_{1}\vert\mathcal{I}(T_{1}))(x_1)\otimes f_{2}\otimes\dots\otimes f_{d}\; d\lambda_{\vec{x}}
		\\ & = & \int\left (\int f_1d\mu_{[T_1],x_1} \right) \otimes f_{2}\otimes\dots\otimes f_{d}\; d\lambda_{\vec{x}}.
	\end{eqnarray*}
	Remark that  the function $\left (\int f_1\;d\mu_{[T_1],x_1} \right)$ is constant $\mu_{[T_1],x_1}$-a.e. and thus continuous $\mu_{[T_1],x_1}$-a.e. Using again the definition of $\lambda_{\vec{x}}$ and the induction hypothesis we get that 
	\begin{eqnarray*}
	\int f_{1}\otimes f_{2}\otimes\dots\otimes f_{d}\; d\lambda_{\vec{x}} & = & \left (\int f_1\;d\mu_{[T_1],x} \right)\left (\int f_2\otimes\cdots\otimes f_d\;d\mu_{[T_2,\ldots,T_d],x_{\ast}} \right)
		\\  & = &  \mathbb{E}(f_{1}\vert\mathcal{I}(T_{1}))(x_1)\mathbb{E}(f_{2}\vert\mathcal{I}(T_{2}))(x_2)\cdot\ldots\cdot\mathbb{E}(f_{d}\vert\mathcal{I}(T_{d}))(x_d).
	\end{eqnarray*}
	By the density of linear combinations of functions  $f_1\otimes \cdots \otimes f_d$, $f_i\in \mathcal{C}_i$ in $C(X_1\times\cdots \times X_d)$, we get that $\lambda_{\vec{x}}=\mu_{[T_1,\ldots,T_d],\vec{x}}$.   
\end{proof}

We are now ready to prove Theorem~\ref{C:main} via Theorem~\ref{Thm:LimitMeasure}. 

\begin{proof}[Proof of Theorem~\ref{C:main}] 
	We keep the notation as above, and let $\mathcal{C}_i$ be a countable family of continuous functions which is dense in $L^{2}(\mu_i)$ for $1\leq i \leq d$. 
	For $N\in\mathbb{N}$ denote $$A_N(x_1,\ldots,x_d) :=\frac{1}{N}\sum_{n=0}^{N-1} f_1\Big(T_{1}^{[a_1(n)]}x_1\Big)\cdot\ldots\cdot f_d\Big( T_{d}^{[a_d(n)]}x_d\Big).$$
	For $k\in\mathbb{N}$, pick $\widehat{f}_{i,k}\in\mathcal{C}_i$  such that $\| f_i - \widehat{f}_{i,k}\|_{2}\leq \frac{1}{k}$ for $1\leq i \leq d$  and denote $$\widehat{A}_{N,k}(x_1,\ldots,x_d):=\frac{1}{N}\sum_{n=0}^{N-1} \widehat{f}_{1,k}\Big(T_{1}^{[a_1(n)]}x_1\Big)\cdot\ldots\cdot \widehat{f}_{d,k}\Big(T_{d}^{[a_d(n)]}x_d\Big).$$ By the definition of $\mathcal{S}^{*}$ and the telescoping inequality,  we have that 
	\[\limsup_{N\to \infty } |A_N(x_1,\ldots,x_d)-\widehat{A}_{N,k}(x_1,\ldots,x_d)| \leq \sum_{i=1}^d  \mathbb{E}(|f_i-\widehat{f}_{i,k}|\vert \mathcal{I}(T_i))(x_i)  \]	
	for $\nu$-a.e. $\vec{x}=(x_1,\ldots,x_d)\in X_1\times\cdots\times X_d$ and for all $k\in \mathbb{N}$. 
	So for every $k\in \mathbb{N}$, 
	\[ \limsup_{N\to \infty } \left |A_N(x_1,\ldots,x_d)- 	\int f_1\otimes \cdots \otimes f_d\; d\mu_{[T_1,\ldots,T_d],\vec{x}} \right |\]
	is bounded by the sum of the terms 
	\begin{equation} \label{eq:AN}
	\limsup_{N\to \infty }   \left |A_N(x_1,\ldots,x_d)-\widehat{A}_{N,k}(x_1,\ldots,x_d) \right |,
	\end{equation}  
\begin{equation}  \label{eq:limitcontinuous}
	\limsup_{N\to \infty } \left  |\widehat{A}_{N,k}(x_1,\ldots,x_d)-	\int \widehat{f}_{1,k}\otimes \cdots \otimes\widehat{f}_{d,k}\; d\mu_{[T_1,\ldots,T_d],\vec{x}} \right |;\;\;\;\textit{and}
	\end{equation} 
	\begin{equation}\label{eq:DifMu}
	\left |\int \widehat{f}_{1,k}\otimes \cdots \otimes\widehat{f}_{d,k}\; d\mu_{[T_1,\ldots,T_d],\vec{x}}-	\int f_1\otimes \cdots \otimes f_d\; d\mu_{[T_1,\ldots,T_d],\vec{x}} \right  |.
	\end{equation} 
	By Theorem \ref{Thm:LimitMeasure}, the Term~\eqref{eq:limitcontinuous} is equal to 0. Again by telescoping, the sum of the Terms~\eqref{eq:AN} and ~\eqref{eq:DifMu} is bounded by $2\sum_{i=1}^d  \mathbb{E}(|f_i-\widehat{f}_{i,k}|\vert \mathcal{I}(T_i))(x_i)$, for $\nu$-a.e. $(x_1,\ldots,x_d)\in X_1\times\cdots\times X_d$.  
	If $A_{m}$ denotes the set of points $ \vec{x}=(x_1,\ldots,x_d)$ such that \[\limsup_{N\to \infty } \left |A_N(x_1,\ldots,x_d)- \int f_1\otimes \cdots \otimes f_d\; d\mu_{[T_1,\ldots,T_d],\vec{x}} \right | \geq \frac{1}{m},\]
	then Markov's inequality implies that for every $m\in \mathbb{N}$ the measure of $A_{m}$ is smaller than $\frac{2d m}{k}$. Since $k$ is arbitrary, we have that $\nu(A_m)=0$ and $\nu(\cap_{m\in\mathbb{N}} A_{m}^c)=1$. It is immediate to check that $A_N(x_1,\ldots,x_d)$ converges for every $(x_1,\ldots,x_d)\in  \cap_{m\in\mathbb{N}} A_{m}^c$. 
\end{proof}

\begin{proof}[Proof of Theorem~\ref{D:sUniqueErgodic}]
Let $x\in X$ and $\lambda_{N,x}=\frac{1}{N}\sum_{n=0}^{N-1}T^{[a(n)]}\delta_x$.  	By Lemma~\ref{lem:inv}, we have that any weak limit of $\lambda_{N,x}$ is $T$-invariant and hence equal to $\mu$ by unique ergodicity. Therefore, $\lambda_{N,x}$ converges to $\mu$ as $N\to\infty$ and the conclusion follows.
\end{proof}

\begin{proof}[Proof of Theorem~\ref{D:mainUniqueErgodic}]
It is the same proof as of Theorem~\ref*{Thm:LimitMeasure}, combined with the fact that the single average converges by Theorem~\ref{D:sUniqueErgodic}.
\end{proof}


\begin{proof}[Proof of Theorem~\ref{Thm:Linear}]
	\textbf{(i) Case where $a_1(n)=pn,$ $ p\in\mathbb{Z}\setminus\{0\}$.} The proof is very similar to the one of Theorem \ref{Thm:LimitMeasure}; we sketch it for completeness.  We will give an expression for any weak limit of the average of the Dirac measure $\delta_{\vec{x}}$, $\vec{x}=(x_1,\ldots,x_d)$ in a dense family of functions, and then is routine to arrive at the conclusion (as in the proof of Theorem~\ref{C:main}). 
	
	
	 Let $\pi_{1}$ and $\pi_2$ be the projections as in the proof of Theorem~\ref{Thm:LimitMeasure}. 
	 By Birkhoff's ergodic theorem we have that $(\pi_{1})_{*}\lambda_{\vec{x}}=\mu_{[T_{1}^{p}],x_1}$ for $\nu$-a.e. $\vec{x}\in X_1\times\cdots\times X_d$. By Theorem~\ref{Thm:LimitMeasure}, $(\pi_{2})_{*}\lambda_{\vec{x}}=\mu_{[T_{2},\dots,T_{d}],x_{\ast}}$ for $\nu$-a.e. $\vec{x}\in X_1\times\cdots\times X_d$. 
	
	Fix a countable dense set of continuous functions $\mathcal{C}_i=\{g_{i,k}: k\in \mathbb{N} \} \subseteq C(X_i)$. We have that there exists a set of full measure $X_1'\subseteq X_1$ such that   $\left(\frac{1}{N}\sum_{n=0}^{N-1}T_{1}^{pn}g_{1,k}(x_1)\right)_N$ converges to $\mathbb{E}(g_{1,k}\vert \mathcal{I}({T^{p}_1}))(x_1)=\int g_{1,k}\; d\mu_{[T^{p}_1],x_1}$ as $N\to\infty$ for every $k\in \mathbb{N}$ and $x_{1}\in X'_1$. Since $\mu_1=\int \mu_{[T^{p}_1],x_{1}}\;d\mu_1(x_{1})$, the same is true for $\mu_{[T_1^{p}],x_1}$-a.e. $y \in X_1$ for $\mu_1$-a.e. $x_1\in X_1$. 
	Let $f_i \in \mathcal{C}_i$, $1\leq i\leq d$ and $\vec{x}\in \pi_1^{-1}(X_1')$.  Applying Lemma \ref{lem:inv1}  we have (as in Theorem~\ref{Thm:LimitMeasure})
		\begin{eqnarray*} \int f_{1}\otimes f_{2}\otimes\dots\otimes f_{d}\; d\lambda_{\vec{x}}
		 & = & \lim_{N\to\infty}\frac{1}{N}\sum_{n=0}^{N-1}\int T_{1}^{pn}f_{1}\otimes f_{2}\otimes\dots\otimes f_{d}\; d\lambda_{\vec{x}}\\
		 & = &  \int \left (\int f_1\;d\mu_{[T^{p}_1],x_1} \right) \otimes f_{2}\otimes\dots\otimes f_{d}\; d\lambda_{\vec{x}}.
		 \end{eqnarray*}
	Remark that  the function $\left (\int f_1\;d\mu_{[T^{p}_1],x_1} \right)$ is constant $\mu_{[T^{p}_1],x_1}$-a.e. and thus continuous $\mu_{[T^{p}_1],x_1}$-a.e. Using again the definition of $\lambda_{\vec{x}}$ and Theorem \ref{Thm:LimitMeasure}, we get that 
	\begin{eqnarray*}
		\int f_{1}\otimes f_{2}\otimes\dots\otimes f_{d}\; d\lambda_{\vec{x}} & = & \left (\int f_1\;d\mu_{[T^{p}_1],x_1} \right)\left (\int f_2\otimes\cdots\otimes f_d\;d\mu_{[T_2,\ldots,T_d],x_{\ast}} \right)
		\\  & = &  \mathbb{E}(f_{1}\vert\mathcal{I}(T^{p}_{1}))(x_1)\mathbb{E}(f_{2}\vert\mathcal{I}(T_{2}))(x_2)\cdot\ldots\cdot\mathbb{E}(f_{d}\vert\mathcal{I}(T_{d}))(x_d).
	\end{eqnarray*}
	By the density of linear combinations of functions  $f_1\otimes \cdots \otimes f_d$, $f_i\in \mathcal{C}_i$ in $C(X_1\times\cdots\times X_d)$, we get that $\lambda_{\vec{x}}=\mu_{[T^{p}_1,T_2,\ldots,T_d],\vec{x}}$.   
	
	\textbf{(ii) Case where $a_1(n)=pn+\ell,$ $ p\in\mathbb{Z}\setminus\{0\}$.}
	Using Case (i), we have
	\begin{equation}\nonumber
	\begin{split}
	   &\quad\int f_{1}\otimes f_{2}\otimes\dots\otimes f_{d}\; d\lambda_{x}
	   \\&=\lim_{N\to\infty}\frac{1}{N}\sum_{n=0}^{N-1}T_{1}^{[pn+\ell]}f_{1}(x_1) T_{2}^{[a_2(n)]}f_{2}(x_2)\cdot\ldots\cdot T_{d}^{[a_d(n)]}f_{d}(x_d)
	   \\&=\lim_{N\to\infty}\frac{1}{N}\sum_{n=0}^{N-1}T_{1}^{pn}\Big(T_{1}^{[\ell]}f_{1}\Big)(x_1) T_{2}^{[a_2(n)]}f_{2}(x_2)\cdot\ldots\cdot T_{d}^{[a_d(n)]}f_{d}(x_d)
	   \\&=\mathbb{E}\Big(T_{1}^{[\ell]}f_{1}\vert\mathcal{I}(T^{p}_{1})\Big)(x_1)\mathbb{E}(f_{2}\vert\mathcal{I}(T_{2}))(x_2)\cdot\ldots\cdot\mathbb{E}(f_{d}\vert\mathcal{I}(T_{d}))(x_d).
	\end{split}
	\end{equation}
	
	\textbf{(iii) Case where $a_1(n)=kn+\ell,$ $ k=p/q\in\mathbb{Q}\setminus\{0\}$.} 	Using Case (ii), we have
	\begin{equation}\nonumber
		\begin{split}
		&\quad\int f_{1}\otimes f_{2}\otimes\dots\otimes f_{d}\; d\lambda_{\vec{x}}
		\\&=\lim_{N\to\infty}\frac{1}{N}\sum_{n=0}^{N-1}T_{1}^{[kn+\ell]}f_{1}(x_1) T_{2}^{[a_2(n)]}f_{2}(x_2)\cdot\ldots\cdot T_{d}^{[a_d(n)]}f_{d}(x_d)
		\\&=\frac{1}{q}\sum_{j=0}^{q-1}\lim_{N\to\infty}\frac{1}{N}\sum_{n=0}^{N-1}T_{1}^{[k(qn+j)+\ell]}f_{1}(x_1) T_{2}^{[a_2(qn+j)]}f_{2}(x_2)\cdot\ldots\cdot T_{d}^{[a_d(qn+j)]}f_{d}(x_d)
		\\&=\frac{1}{q}\sum_{j=0}^{q-1}\mathbb{E}\Big(T_{1}^{[\frac{pj}{q}+\ell]}f_{1}\vert\mathcal{I}(T^{p}_{1})\Big)(x_1)\mathbb{E}(f_{2}\vert\mathcal{I}(T_{2}))(x_2)\cdot\ldots\cdot\mathbb{E}(f_{d}\vert\mathcal{I}(T_{d}))(x_d).
		\end{split}
	\end{equation}
	

\textbf{(iv) Case where $a_1(n)=\gamma n+\ell,$ $\gamma\in \mathbb{R} \setminus \mathbb{Q}$} (recall the notations of Corollary~\ref{Cor:irequ}).
By Case (i), and by passing to the mapping torus extension, we have 
	\begin{equation}\nonumber
	\begin{split}
	&\quad\int f_{1}\otimes f_{2}\otimes\dots\otimes f_{d}\; d\lambda_{x}
	\\&=\lim_{N\to\infty}\frac{1}{N}\sum_{n=0}^{N-1}T_{1}^{[\gamma n+\ell]}f_{1}(x_1) T_{2}^{[a_2(n)]}f_{2}(x_2)\cdot\ldots\cdot T_{d}^{[a_d(n)]}f_{d}(x_d)
	\\&=\lim_{N\to\infty}\frac{1}{N}\sum_{n=0}^{N-1}S^{n}\tilde{f}_{1}\Big(\{\ell\},T^{[\ell]}x_1\Big) T_{2}^{[a_2(n)]}f_{2}(x_2)\cdot\ldots\cdot T_{d}^{[a_d(n)]}f_{d}(x_d)
	\\&=\mathbb{E}\Big(\tilde{f}_{1}\vert\mathcal{I}(S)\Big)\Big(\{\ell\},T^{[\ell]}x_1\Big)\mathbb{E}(f_{2}\vert\mathcal{I}(T_{2}))(x_2)\cdot\ldots\cdot\mathbb{E}(f_{d}\vert\mathcal{I}(T_{d}))(x_d).
	\end{split}
	\end{equation}	
	By Corollary \ref{Cor:irequ}, the result follows.
\end{proof}

 \begin{proof}[Proof of Corollary~\ref{C:main1}] It follows by Theorems~\ref{C:main} and ~\ref{Thm:Linear}.
\end{proof}


\begin{thebibliography}{9999}

\bibitem{Aus} T.~Austin. On the norm convergence of nonconventional ergodic averages. \textit{Ergodic Theory Dynam. Systems} {\bf 30} (2010), no. 2, 321--338.

\bibitem{Be87a} V.~Bergelson. Weakly mixing PET. {\em Ergodic Theory
	Dynam. Systems} \textbf{7} (1987), no. 3, 337--349.

\bibitem{BK} V.~Bergelson and I.~H\r{a}land-Knutson. Weakly mixing implies mixing of higher orders along tempered functions. {\em Ergodic Theory
	Dynam. Systems} \textbf{29} (2009), no. 5, 1375--1416.


\bibitem{BL} V.~Bergelson and A.~Leibman. A nilpotent Roth theorem. {\em Invent. Math.} \textbf{147} (2002), 429--470.

\bibitem{BKQW} M.~Boshernitzan, G.~Koles, A.~Quass and M.~Weirdl. Ergodic averaging sequences. {\em J. Anal. Math.} \textbf{95} (2005), 63--103.

\bibitem{Bo} J.~Bourgain. Double recurrence and almost sure convergence. \textit{ J. Reine Angew. Math.} {\bf 404}
(1990), 140--161.

\bibitem{Bo2} J.~Bourgain. Pointwise ergodic theorems for arithmetic sets. \textit{Publications Mathematiques, Institut des Hautes Etudes Scientifiques.} {\bf 69}  (1989), 5--45.


\bibitem{CFH} Q.~Chu, N.~Frantzikinakis and  B.~Host. Ergodic averages of commuting transformations with distinct degree polynomial iterates. {\em Proc. of the London Math. Society.} (3), \textbf{102} (2011), 801--842.

\bibitem{DS0} S.~Donoso and W.~Sun. Pointwise multiple averages for systems with two commuting transformations, to appear {\em  Ergodic Theory Dynam. Systems}.

\bibitem{DS} S.~Donoso and W.~Sun. Pointwise convergence of some multiple ergodic averages,  arXiv:1609.02529. 


\bibitem{F2} N.~Frantzikinakis.
Multiple recurrence and convergence for Hardy field sequences of polynomial growth. 
{\em J. d'Analyse Math.} \textbf{112} (2010), 79--135.   

\bibitem{F3} N.~Frantzikinakis.
A multidimensional Szemeredi theorem for Hardy sequences of polynomial growth. 
{\em Tran. of the A. M. S.} \textbf{367}, no. 8, (2015), 5653--5692.


\bibitem{Fu}
H.~Furstenberg. Ergodic behavior of diagonal measures and a theorem of Szemer\'{e}di on arithmetic progressions. {\em J. Analyse Math.} \textbf{31} (1977), 204--256.

\bibitem{GHSY} Y.~Gutman, W.~Huang, S.~Shao and X.~Ye.  Almost sure convergence of the multiple ergodic average for certain weakly mixing systems. https://arxiv.org/abs/1612.02873. 


\bibitem{HK05} B.~Host and B.~Kra. Nonconventional averages and
nilmanifolds. \textit{Ann. of Math. (2)} {\bf 161} (2005), no. 1, 398--488.


\bibitem{H} B.~Host, Ergodic seminorms for commuting transformations and applications. \textit{ Studia Math. } {\bf 195 } (2009), no. 1, 31--49.



\bibitem{K} A.~Koutsogiannis. Integer part polynomial correlation sequences. {\em Ergodic
	Theory Dynam. Systems} 1--18. doi:10.1017/etds.2016.67.

\bibitem{HSY} W.~Huang, S.~Shao, X.~Ye. Pointwise convergence of multiple ergodic averages and strictly ergodic models.  arXiv:1406.5930.

\bibitem{Tao} T.~Tao. Norm convergence of multiple ergodic averages for commuting transformations. \textit{Ergodic Theory and Dynamical Systems} {\bf 28 } (2008), no. 2, 657--688.

\bibitem{Walsh12} M.~Walsh. Norm convergence of nilpotent ergodic averages. {\em Annals of Mathematics} \textbf{175} (2012), no. 3, 1667--1688.

\end{thebibliography}
\end{document}